\documentclass[12pt,amsfonts]{article}
\usepackage[margin=1in]{geometry}  
\usepackage{graphicx}              
\usepackage[latin1]{inputenc}
\usepackage{amsmath, amssymb, amsthm, epsfig}               
\usepackage{enumerate}
\usepackage{amsfonts}              
\usepackage{amsthm}                
\usepackage{amsthm}
\usepackage{enumerate}
\theoremstyle{plain}
\usepackage{color}
\def\dis{\displaystyle}
\def\nd{\noindent}
\def\thend{\rule{3mm}{3mm}}

\newtheorem{claim}{Claim}[section]
\newtheorem{theorem}{Theorem}[section]

\newtheorem{lemma}{Lemma}[section]

\newtheorem{definition}{Definition}[section]
\newtheorem{corollary}{Corollary}[section]
\newtheorem*{theorem*}{Theorem}

\numberwithin{equation}{section}

\begin{document}
\title{Existence and concentration of positive solutions for a logarithmic Schr\"{o}dinger  equation via penalization method}
\author{ Claudianor O. Alves\footnote{C.O. Alves was partially supported by CNPq/Brazil  304804/2017-7.} \,\, and \,\, Chao Ji \footnote{C. Ji was partially supported by Shanghai Natural Science Foundation(18ZR1409100).}}

\maketitle

\begin{abstract}
In this  article we are concerned with the following logarithmic Schr\"{o}dinger equation
$$
\left\{
\begin{array}{lc}
-{\epsilon}^2\Delta u+ V(x)u=u \log u^2, & \mbox{in} \,\, \mathbb{R}^{N}, \\
u \in H^1(\mathbb{R}^{N}), & \;  \\
\end{array}
\right.
$$
where $\epsilon >0, N \geq 1$  and $V:\mathbb{R}^{N}\rightarrow \mathbb{R}$ is a continuous potential. Under a local assumption on the potential $V$, we use
the variational methods to prove the existence and concentration of  positive solutions  for the  above problem.
\end{abstract}



\section{Introduction}
 In the past few decades, the nonlinear elliptic equation
  $$
-{\epsilon}^2\Delta u+ V(x)u=f(u),  \,\, x\in \mathbb{R}^{N},
\eqno{(S_\epsilon)}
 $$
 where $N\geq 1$, $\epsilon>0$ is a positive parameter, $V, f$ are continuous functions verifying some assumptions, has been studied by many researchers.
  A basic motivation for the study of problem $(S_\epsilon)$ is to seek for the standing waves of the following nonlinear Schr\"odinger equation
$$
i \epsilon \frac{\partial  \Psi}{\partial t} = - \epsilon^{2}\Delta \Psi + (V(x	)+E)\Psi- f(\Psi), \quad\text{for}\,\, x\in \mathbb{R}^{N},
\eqno{(NLS)}
$$
namely, solution of the form $\Psi(x, t)=\exp(-i Et/\epsilon)u(x)$ with $u(x)$ is a real value function.
There is a broad literature on the existence and concentration of positive solutions for general
semilinear elliptic equations $(S_\epsilon)$ for the case $N\geq 1$, see for example, Floer and Weinstein \cite{fw}, Oh \cite{o1, o2}, Rabinowitz \cite{r}, Wang \cite{W},
Cingolani and Lazzo \cite{CL}, Ambrosetti, Badiale and Cingolani \cite{ABC},  Gui \cite{G}, del Pino and Felmer \cite{DF} and their references.

 In \cite{r}, by a variant of a mountain pass argument,  Rabinowitz proved the existence of positive solutions of problem $(S_\epsilon)$ for $\epsilon>0$ small, whenever
 $$
 V_{\infty}=\underset{\vert x\vert\rightarrow\infty}{\lim\inf}\, V(x)>\inf_{x\in \mathbb{R}^{N}} V(x)=V_{0}>0.
 $$
  Later, Wang \cite{W} used variational methods to show that these solutions concentrate at global minimum points of $V$ as $\epsilon\rightarrow 0$.
  In \cite{DF}, del Pino and Felmer  found solutions which concentrate around local minimum of $V$ by introducing a penalization method. More precisely,
  they assumed that there is an open and bounded set $\Lambda\subset \mathbb{R}^{N}$ such that
   $$
       0<V_{0}=\inf_{z\in\Lambda}V(z)<\min_{z\in\partial \Lambda}V(z).
   $$
In the above-mentioned papers, the authors assumed that the nonlinearity $f$ satisfies superlinear, subcritical growth conditions and the well-known Ambrosetti-Rabinowitz condition, this allow us to employ the variational methods for the class of $C^{1}$ functional to attach these problems.

  Recently, the logarithmic Schr\"odinger equation given by
$$
i \epsilon \partial_t \Psi = - \epsilon^{2}\Delta \Psi + (W(x	)+w)\Psi- \Psi \log|\Psi|^2, \ \Psi:[0, \infty) \times \mathbb{R}^{N} \rightarrow \mathbb{C}, \ N \geq 1,
$$
 has  also  received considerable attention. This class of equation has some important physical applications, such as quantum mechanics, quantum optics, nuclear physics, transport and diffusion phenomena, open quantum systems, effective quantum gravity, theory of superfluidity and Bose-Einstein condensation (see \cite{z} and the references therein). In its turn, standing waves solution, $\Psi$, for this logarithmic Schr\"odinger equation is related
 to the solutions of the equation
 $$
-{\epsilon}^2\Delta u+ V(x)u=u \log u^2, \,\, \mbox{in} \quad \mathbb{R}^{N}.
\eqno{(P_\epsilon)}
$$

Besides the importance in applications, the equation $(P_\epsilon)$ also raises many difficult mathematical problems.  The natural candidate for the associated energy functional would formally be the functional
\begin{equation}\label{lula}
\widehat{I}_\epsilon(u)=\frac{1}{2}\displaystyle \int_{\mathbb{R}^N}(\epsilon^{2}|\nabla u|^2+(V(x)+1)|u|^2)dx-\displaystyle \frac{1}{2}\int_{\mathbb{R}^N} u^2 \log u^2\,dx.
\end{equation}
It is easy to see that each critical point of $\widehat{I}_\epsilon$ is a solution of (\ref{lula}).
However, this functional is not well defined in $H^{1}(\mathbb{R}^N)$ because there is $u \in H^{1}(\mathbb{R}^N)$ such that $\int_{\mathbb{R}^N}u^{2}\log u^2 \, dx=-\infty$. In order to overcome this technical difficulty some authors have used different techniques to study the existence, multiplicity and concentration of the solutions under some assumptions on the potential $V(x)$, which can be seen in  \cite{AlvesdeMorais}, \cite{AlvesdeMoraisFigueiredo}, \cite{AlvesChao}, \cite{AlvesChao1}, \cite{squassina}, \cite{ASZ}, \cite{DZ}, \cite{cs},  \cite{sz}, \cite{sz2}, \cite{TZ},  \cite{WZ} and the references therein. In \cite{ASZ}, different from the previous contribution on the this subject, the authors directly faced the loss of the compactness and studied the existence of multiple solutions by using non-smooth critical point theory, which was also used in \cite{Squ} to establish the existence and concentration of the solutions for the quasi-linear elliptic equations. We also notice that the soliton dynamics behaviour for the logarithmic Schr\"odinger equations were studied by some mathematicians, see for example \cite{AS}. This class of problems are not fully solved as it depends on the regularity property of the functional.

In a recent paper \cite{AlvesdeMorais}, Alves and de Morais Filho established the existence and concentration of positive solutions to problem $(P_\epsilon)$, for $\epsilon>0$, by requiring that $V$ verifies the global assumption introduced by Rabinowitz \cite{r}
\begin{equation}\label{ps1}
V_{\infty}:=\lim_{\vert x\vert\rightarrow \infty}\, V(x)> \inf_{x\in \mathbb{R}^N}\, V(x)=V_{0}>-1.
\end{equation}

Later, Alves and Ji   \cite{AlvesChao} considered the multiple positive solutions to problem $(P_\epsilon)$  under the same assumption (\ref{ps1}). More precisely, it was proved that the "shape" of the graph of the function $V$ affects  the number of nontrivial solutions.

 It is quite natural to consider the existence and concentration results of the solutions for problem $(P_\epsilon)$ when the potential $V$ satisfies a local assumption.    Inspired by \cite{AlvesdeMorais, DF, szulkin},  the main purpose
of this paper is to investigate the existence and concentration of positive solutions of problem $(P_\epsilon)$ by combining a local assumption on $V$ and adapting the penalization method found in del Pino and Felmer \cite{DF}.

Throughout the paper, we make the following assumptions on the potential $V$:
\begin{itemize}
	\item[\rm ($V1$)] $V\in C(\mathbb{R}^N, \mathbb{R})$ and $\displaystyle \inf_{x\in\mathbb{R}^N}V(x)=V_{0}>-1$;

	\item[\rm ($V2$)]There exists an open and bounded set $\Lambda\subset \mathbb{R}^{N}$ satisfying
       $$
       -1<V_{0}=\inf_{x\in\Lambda}V(x)<\min_{x\in\partial \Lambda}V(x).
       $$
	
\end{itemize}

By a change of variable, we know that problem $(P_\epsilon)$ is equivalent to the problem
\begin{equation}\label{1'}
\left\{
\begin{array}{lc}
-\Delta u+ V(\epsilon x)u=u \log u^2, & \mbox{in} \quad \mathbb{R}^{N}, \\
u \in H^1(\mathbb{R}^{N}). & \;  \\
\end{array}
\right.
\end{equation}
\begin{definition}
For us, a positive solution of (\ref{1'}) means a positive function $u \in H_{\epsilon}$ such that $ u^2\log u^2 \in L^1(\mathbb{R}^{N})(i.e., J_{\epsilon}(u)<\infty)$  and
\begin{equation}
\displaystyle \int_{\mathbb{R}^N} (\nabla u  \nabla v +V(\epsilon x)u v)dx=\displaystyle \int_{\mathbb{R}^N} uv \log u^2dx,\,\, \mbox{for all }\, v \in C^\infty_0(\mathbb{R}^{N}).
\end{equation}
\end{definition}

We shall use the variational method found in Szulkin \cite{szulkin} to prove the existence of nontrivial solutions for problem (\ref{1'}). Here, we will show that any critical point of the associated energy functional
$$
J_{\epsilon}(u)=\dis\frac{1}{2}\int_{\mathbb{R}^N} \big (|\nabla u|^2+(V(\epsilon x) +1)|u|^2\big)dx-\dis\frac{1}{2}\int_{\mathbb{R}^N} u^2\log u^2dx,
$$
in the sub-differential sense, is a weak solution of (\ref{1'}) in $H^1(\mathbb{R}^{N})$. Aiming this approach, let us define the Banach space
$$
H_{\epsilon}:=\left\{u\in H^{1}(\mathbb{R}^N): \int_{\mathbb{R}^N} V(\epsilon x)\vert u\vert^{2}dx<\infty\right\}
$$
endowed with the norm
$$
\Vert u\Vert_{\epsilon}=\Big(\int_{\mathbb{R}^N} (\vert \nabla u\vert^{2}+(V(\epsilon x)+1)\vert  u\vert^{2})dx\Big)^{1/2}.
$$

The main result of this paper is the following:

\begin{theorem}\label{teorema}
Suppose that $V$ satisfies $(V1)-(V2)$. Then, there exists $\epsilon_{0}>0$ such that, for any $\epsilon\in (0, \epsilon_{0})$, the problem $(P_\epsilon)$
 has a positive solution $v_{\epsilon}$. Moreover, if $\eta_{\epsilon}\in \mathbb{R}^{N}$ is a global maximum point of $v_{\epsilon}$,
 we have
 $$
 \lim_{\epsilon\rightarrow 0}\,V(\eta_{\epsilon})=V_{0}.
 $$
\end{theorem}

The proof of Theorem \ref{teorema} is inspired from \cite{AlvesdeMorais, DF, szulkin}, however we are working with the logarithmic Schr\"{o}dinger equation, whose the energy functional associated is not continuous, for this reason,  some estimates  for this problem are also very delicate and different from those used in the Schr\"{o}dinger equation $(S_\epsilon)$. Also for this reason, we shall modify the nonlinearity in a special way to work with a modified problem. Making
some estimates we prove that the solutions obtained for the modified problem are solutions of the original problem when $\epsilon>0$ is sufficient small.
On the other hand, a equality of the type $J_{\epsilon}(u)-\frac{1}{2}J'_{\epsilon}(u)u=\frac{1}{2}\int_{\mathbb{R}^N} |u|^{2}\, dx$ is very important for the study of the  logarithmic Schr\"odinger equations,  for example, in \cite{sz, cs}, the authors  used it and the logarithmic Sobolev inequality to verify the boundedness of $(PS)$ sequence. But,  the functional associated with the modified problem doesn't satisfy the equality above, so the proof of the boundedness of $(PS)$ sequence is a great challenge, and here we developed a new way to prove this boundedness, see Lemmas \ref{LI} and \ref{boundedness} for more details. Moreover, since the functional associated with the modified problem also lost some other good properties, it is difficult to verify the mountain pass geometry, see Lemma \ref{vicente}. The reader is invited to see that the way how we attach these problems in Section 3 is different of that explored in \cite{AlvesdeMorais, cs, sz}. After our paper was completed, we learned of some related work for the problem (\ref{1'}), see \cite{ZZ}. In that paper, the authors also considered the positive bound state solutions exist and concentrate as $\epsilon\rightarrow 0$ under a local assumption on the potential $V$. However, the approach of the present paper is completely different from one in \cite{ZZ}, and our method can be applied for the study of other problems in this field, for example, see \cite{AlvesChao3}. The plan of the paper is as follows: In Section 2 we show some preliminary results which can be used later on. In Section 3 we study the modified problem, this is a key point in our approach. Finally, in Section 4, we give the proof of Theorem \ref{teorema}.

\vspace{0.5 cm}

\noindent \textbf{Notation:} From now on in this paper, otherwise mentioned, we use the following notations:
\begin{itemize}
	\item $B_r(u)$ is an open ball centered at $u$ with radius $r>0$, $B_r=B_r(0)$.
	
	\item If $g$ is a measurable function, the integral $\int_{\mathbb{R}^N}g(z)dz$ will be denoted by $\int g(z)dz$.
	
	\item   $C$ denotes any positive constant, whose value is not relevant.
	
	\item  $|\,\,\,|_p$ denotes the usual norm of the Lebesgue space $L^{p}(\mathbb{R}^N)$, for $p \in [1,+\infty]$,
     $\Vert\,\,\,\Vert$ denotes the usual norm of the Sobolev space $H^{1}(\mathbb{R}^N)$.

	\item  For the measurable set $A\subset \mathbb{R}^N$, $\vert A\vert$ denotes the Lebesgue measure of the set $A$.
	
	\item  $H_c^{1}(\mathbb{R}^N)=\{u \in H^{1}(\mathbb{R}^N)\,:\, u \,\, \mbox{has compact support}\, \}.$
	
	\item $o_{n}(1)$ denotes a real sequence with $o_{n}(1)\to 0$ as $n \to +\infty$.
	
	\item $2^*=\frac{2N}{N-2}$ if $N \geq 3$ and $2^*=+\infty$ if $N=1, 2$.
\end{itemize}

\section{Preliminaries}\label{tpm}
Let us go back to the functional $J_{\epsilon}$. Following the approach explored in \cite{AlvesdeMorais,cs,sz}, due to the lack of smoothness of $J_{\epsilon}$, let us decompose it into a sum of a $C^1$ functional plus a convex lower semicontinuous functional, respectively. For $\delta>0$, let us define the following functions:

$$
F_1(s)=
\left\{ \begin{array}{lc}
0, & \; s= 0 \\
-\frac{1}{2}s^2\log s^2 & \; 0<\vert s\vert <\delta \\
-\frac{1}{2}s^2(
\log\delta^2+3)+2\delta|s|-\frac{1}{2}\delta^2,  & \; \vert s \vert \geq \delta
\end{array} \right.
$$

\noindent and
\begin{equation}\label{efes}
F_2(s)=
\frac{1}{2}s^2\log s^2+ F_1(s), \quad    s\in \mathbb{R}.
\end{equation}
It was proved in \cite{cs} and \cite{sz} that $F_1$ and $F_2$ verify the following properties:
\begin{equation}\label{eq1}
F_1, F_2 \in C^1(\mathbb{R},\mathbb{R}).
\end{equation}
If $\delta >0$ is small enough, $F_1$ is convex, even, $F_1(s)\geq 0$ for all $ s\in \mathbb{R}$ and
\begin{equation}\label{eq2}
F'_1(s)s\geq 0, \ s \in \mathbb{R}.
\end{equation}
For each fixed $p \in (2, 2^*)$, there is $C>0$ such that
\begin{equation}\label{eq5}
|F'_2(s)|\leq C|s|^{p-1}, \quad \forall s \in \mathbb{R}.
\end{equation}
Let us define
\begin{equation}\label{fi}
\Phi_\epsilon(u)=\frac{1}{2}\int_{\mathbb{R}^N} (|\nabla u|^{2}+(V(\epsilon x)+1))|u|^{2})\,dx-\displaystyle \int_{\mathbb{R}^N} F_2(u)\,dx,
\end{equation}
\noindent and
\begin{equation}\label{psi}
\Psi(u)=\displaystyle \int_{\mathbb{R}^N} F_1(u)\,dx,
\end{equation}
then
\begin{equation}\label{funcional}
J_\epsilon(u)=\Phi_\epsilon(u)+\Psi(u), \quad u \in H_{\epsilon}.
\end{equation}
Using the above information, it follows that $\Phi_\epsilon \in C^{1}(H^1(\mathbb{R}^{N}),\mathbb{R})$, $\Psi$ is convex and lower semicontinuous, but $\Psi$ is not a $C^1$ functional, since we are working on $\mathbb{R}^N$. Due to this fact, we will look for a critical point in the sub-differential. Here we
state some definitions that can be found in \cite{szulkin}.

\begin{definition} \label{definicao}
	
	Let $E$ be a Banach space, $E'$ be the dual space of $E$ and $\langle \cdot,\cdot \rangle$ be the duality paring between $E'$ and $E$. Let $J:E \to \mathbb{R}$ be a functional of the form $J(u)=\Phi(u)+\Psi(u)$, where $\Phi \in C^{1}(E,\mathbb{R})$ and $\Psi$ is convex and lower semicontinuous. Let us list some definitions:

	\noindent (i) The sub-differential $\partial J(u)$ of the functional $J$ at a point $u \in E$ is the following set
	
	\begin{equation}\label{subdif}
	\{w \in E': \langle \Phi'(u), v-u \rangle+\Psi(v)-\Psi(u)\geq \langle w, v-u\rangle, \ \forall v \in E\}.
	\end{equation}
	
	\noindent (ii) A critical point of $J$ is a point $u \in E$ such that $J(u)< + \infty$ and
	$0 \in \partial J(u)$,\textit{ i.e.}
	
	\begin{equation}\label{critical}
	\langle \Phi'(u), v-u\rangle+\Psi(v)-\Psi(u)\geq 0, \ \forall v \in E.
	\end{equation}
	
	\noindent (iii) A Palais-Smale sequence at level $d$ for $J$ is a sequence $(u_n)\subset E$ such that $J(u_n)\rightarrow d$ and there is a numerical sequence $\tau_n \rightarrow 0^+$ with
	\begin{equation}\label{psequence}
	\langle \Phi'(u_n), v-u_n\rangle+\Psi(v)-\Psi(u_n)\geq -\tau_n||v-u_n||, \ \forall v \in E.
	\end{equation}

	\noindent (iv) The functional $J$ satisfies the Palais-Smale condition at level $d$ $((PS)_d$ condition, for short$)$ if all Palais-Smale sequences at level $d$ has a convergent subsequence.
	
	\noindent (v) The effective domain of $J$ is the set $D(J)=\{u \in E: J(u)< +\infty\}.$
	
\end{definition}
To proceed further we gather and state below some useful results that leads to a better understanding of the problem and of its particularities. In what follows, for each $u \in D(J_\epsilon)$, we set the functional $J'_\epsilon(u):H_c^{1}(\mathbb{R}^N) \to \mathbb{R}$ given by
$$
\langle J'_\epsilon(u),z\rangle=\langle \Phi_\epsilon'(u),z\rangle-\int F'_1(u)z\,dx, \quad \forall z \in H_c^{1}(\mathbb{R}^N)
$$
and define
$$
\|J'_\epsilon(u)\|=\sup\left\{\langle J'_\epsilon(u),z\rangle\,:\, z \in H_c^{1}(\mathbb{R}^N) \quad \mbox{and} \quad \|z\|_\epsilon \leq 1 \right\}.
$$
If $\|J'_\epsilon(u)\|$ is finite, then $J_\epsilon'(u)$ may be extended to a bounded operator in $H_{\epsilon}$, and so,  it can be seen as an element of $H_{\epsilon}'$.

\begin{lemma} \label{lema} Let $J_\epsilon$ satisfy  (\ref{funcional}), then:\\
	\noindent (i) If $u \in D(J_{\epsilon})$ is a critical point of $J_\epsilon$, then
	$$
	\langle \Phi_\epsilon'(u), v-u \rangle +\Psi(v)-\Psi(u)\geq 0, \quad  \forall v \in H_{\epsilon},
	$$
	or equivalently
	$$
	\int \nabla u \nabla (v-u)\, dx + \int (V(\epsilon x)+1)u(v-u)\,dx + \int F_1(v)\, dx - \int F_1(u)\, dx \geq \int F'_2(u)(v-u)\,dx,  \forall v \in H_{\epsilon}.
	$$
	\noindent (ii) For each $u \in D(J_\epsilon)$ such that $\|J'_\epsilon(u)\|< +\infty$, we have $\partial J_\epsilon(u) \not= \emptyset$, that is, there is $w \in  H_{\epsilon}'$, which is denoted by $w=J'_\epsilon(u)$, such that
	$$
	\langle \Phi_\epsilon'(u), v-u \rangle + \int F_1(v)\, dx - \int F_1(u)\, dx \geq \langle w,v-u \rangle, \quad  \forall v \in H_{\epsilon}, \,\, ( \mbox{see} \, \cite{sz2,TZ} )
	$$

	\noindent 	(iii) If a function $u \in D(J_\epsilon)$ is a critical point of $J_\epsilon$, then $u$ is a solution of (\ref{1'}) ( (i) in Lemma 2.4, \cite{cs}).
	
	\noindent 	(iv) If  $(u_n)\subset H_{\epsilon}$ is a Palais-Smale sequence, then
	\begin{equation} \label{ps}
	\langle J'_\epsilon(u_n), z \rangle=o_n(1)\|z\|_\epsilon , \quad \forall z \in H_c^{1}(\mathbb{R}^N).
	\end{equation}
	
	(see (ii) in Lemma 2.4, \cite{cs}).
	
	\noindent	(v) If $\Omega$ is a bounded domain with regular boundary, then $\Psi$ (and hence $J_\epsilon$) is of class $C^1$ in $H^{1}(\Omega)$ (Lemma 2.2 in \cite{sz}). More precisely, the functional
	$$
	\Psi(u)=\int_{\Omega}F_1(u)\,dx, \quad \forall u \in H^{1}(\Omega)
	$$
	belongs to $C^{1}(H^{1}(\Omega),\mathbb{R})$. 	
\end{lemma}

As a consequence of the above proprieties, we have  the following results whose the proofs can be found in \cite{AlvesdeMorais}.
\begin{lemma}
	If $u \in D(J_\epsilon)$ and $\|J'_\epsilon(u)\|< +\infty$, then $F^{'}_1(u)u \in L^1(\mathbb{R}^{N})$.
\end{lemma}

An immediate consequence of the last lemma is the following.

\begin{corollary} \label{C1} For each $u \in D(J_\epsilon) \setminus \{0\}$ with $\|J'_\epsilon(u)\|< +\infty$, we have that
	$$
	J'_\epsilon(u)u=\int (|\nabla u|^{2}+V(\epsilon x)|u|^2)\,dx-\int u^{2} \log u^2\,dx
	$$
	and
	$$
	J_{\epsilon}(u)-\frac{1}{2}J'_{\epsilon}(u)u=\frac{1}{2}\int |u|^{2}\, dx.
	$$

\end{corollary}

\begin{corollary} \label{C2} If $(u_n) \subset H_{\epsilon}$ is a $(PS)$ sequence for $J_\epsilon$, then $J_\epsilon'(u_n)u_n=o_n(1)\|u_n\|_\epsilon$. If $(u_n)$ is bounded, we have
	$$
	J_{\epsilon}(u_n)=J_{\epsilon}(u_n)-\frac{1}{2}J_{\epsilon}'(u_n)u_n+o_n(1)\|u_n\|_\epsilon=\frac{1}{2}\int |u_n|^{2}\, dx+o_n(1)\|u_n\|_\epsilon, \quad \forall n \in \mathbb{N}.
	$$
\end{corollary}

\begin{corollary} \label{C3} If $u \in H_{\epsilon}$ is a critical point of $J_\epsilon$ and $v \in H_{\epsilon}$ verifies $F'_1(u)v \in L^{1}(\mathbb{R}^N)$, then $J'_\epsilon(u)v=0$.
	
\end{corollary}

\section{The modified problem}\label{tpm}
In order to prove our main theorem, we modify problem (\ref{1'}) and then consider the existence of solutions to the modified problem. For our problem,
it is direct to consider $u\log u^{2}+u$ as $f$ appears in \cite{DF}, but it is easy to verify that it does not satisfy the basic assumptions of $f$ that were assumed in \cite{DF},
for example, $t\log t^{2}+t\neq o(t)$ as $t\rightarrow 0$.  Thus, we cannot apply directly del Pino and Felmer's method. By a simple observation,
it is easy to see that
\begin{equation*}
\frac{F'_2(s)}{s}\,\,\, \text{is nondecreasing for}\,\, s>0\quad\text{and}\quad\frac{F'_2(s)}{s}\,\, \text{is strictly increasing for}\,\, s>\delta,
\end{equation*}
\begin{equation*}
\lim_{s\rightarrow +\infty}\frac{F'_2(s)}{s}=+\infty,
\end{equation*}
and
\begin{equation*}
F'_2(s)\geq 0 \,\, \text{for}\,\, s>0\,\, \text{and}\,\,F'_2(s)> 0 \,\, \text{for}\,\, s>\delta.
\end{equation*}
In what follows we need to fix some notations. Let $l>0$ small such that $V_{0}+1\geq 2l$,  $a_{0}>0$ such that $\frac{F'_2(a_{0})}{a_{0}}=l$, it is clear that $a_{0}>\delta$. We define
$$
\tilde{F}'_2(t)=
\left\{ \begin{array}{lc}
F'_2(s)& \; 0\leq s\leq a_{0} \\
ls,  & \; s \geq a_{0}.
\end{array} \right.
$$
If $\chi_{\Lambda}$ denotes the characteristic function of the set $\Lambda$, we introduce the penalized nonlinearity $G_{2}': \mathbb{R}^N\times \mathbb{R}^{+}\rightarrow \mathbb{R}$ by setting
\begin{equation*}
G'_2(x, t)=\chi_{\Lambda}F'_2(t)+(1-\chi_{\Lambda})\tilde{F}'_2(t).
\end{equation*}
Since our attention is to find the positive solutions of problem, we shall consider the following modified problem
$$
-\Delta u+ (V(\epsilon x)+1)u=G'_2(\epsilon x, u^{+})- F'_1(u^{+}),\,\, \mbox{in} \,\, \mathbb{R}^{N}.
\eqno{(P_\epsilon)^{*}}
$$
 We  notice that, if $u_{\epsilon}$ is a positive solution of problem $(P_\epsilon)^{*}$ with $0< u_{\epsilon}(x)\leq a_{0}$ for all $x\in \mathbb{R}^N\backslash \Lambda_{\epsilon}$, then $G'_2(\epsilon x, u_{\epsilon})=F'_2(u_{\epsilon})$ and therefore, $v_{\epsilon}=u_{\epsilon}(\frac{x}{\epsilon})$ is also a solution of $(P_\epsilon)$, where
$$
\Lambda_{\epsilon}:=\{x\in \mathbb{R}^N: \epsilon x\in \Lambda\}.
$$

In what follows, we will look for nontrivial critical points for the functional
\begin{equation*}
I_{\epsilon}(u)=\frac{1}{2}\int \big (|\nabla u|^2+(V(\epsilon x) +1)|u|^2\big)dx+\int F_{1}(u^{+})dx-\int  G_{2}(\epsilon x, u^{+})dx,
\end{equation*}
in the sub-differential sense, where
$$
u^{+}=\max\{u(x), 0\} \quad \mbox{and} \quad G_{2}(x, t)=\int_{0}^{t}G'_{2}(x, s)ds \quad \forall (x, t) \in \mathbb{R}^N\times \mathbb{R}.
$$

Let $H_\epsilon^{+}$ be the open subset of $H_\epsilon$ given by
$$
H_\epsilon^{+}=\{u\in H_\epsilon: \vert \text{supp}(u^{+})\cap \Lambda_{\epsilon}\vert>0\}.
$$
The functional $I_{\epsilon}$ satisfies the mountain pass geometry \cite{Willem}.

\begin{lemma}\label{vicente} For all $\epsilon>0$, the functional $I_{\epsilon}$ satisfies the following conditions:\\
\noindent (i)  $I_{\epsilon}(0)=0$;\\
\noindent (ii) there exist $\alpha, \rho>0$ such that $I_{\epsilon}(u)\geq \alpha$ for any $u\in H_{\epsilon}$ with $\Vert u\Vert_{\epsilon}=\rho$;\\
\noindent (iii) there exists $e\in H_{\epsilon}$ with $\Vert e\Vert_{\epsilon}>\rho$ such that $I_{\epsilon}(e)<0$.
\end{lemma}

\begin{proof} \mbox{} \\
\noindent $(i)$: It is clear.\\
\noindent $(ii)$:
 Note that $I_{\epsilon}(u)\geq\dis\frac{1}{4}\|u\|_\epsilon^2-\int_{\Lambda_{\epsilon}} F_{2}(u^+)dx$. Hence,  from (\ref{eq5}),  fixed $p \in (2,2^*)$, it follows that
$$
I_{\epsilon}(u)\geq \dis\frac{1}{4}\|u\|_\epsilon ^2-C\|u\|_\epsilon^p \geq C_1> 0,
$$
for some $C_1>0$ and $\|u\|_\epsilon>0$ small enough. Here the constant $C_1$ does not depend on $\epsilon$. \\
\noindent $(iii)$:  For each $u\in H_\epsilon^{+}$ and $t>0$. By recalling that
$$
{\mathbb{R}}^N= (\Lambda_\epsilon \cup [tu^+ \leq a_0]) \cup (\Lambda_\epsilon^c \cap [tu^+ > a_0]),
$$
the definition of ${G}_2$ gives
$$
\int F_1(tu^+)\,dx-\int G_2(\epsilon x,tu^+)\,dx\leq -\frac{1}{2}\int_{\Lambda_\epsilon \cup [tu^+ \leq a_0]}|tu^+|^2\log(|tu^+|^2)\,dx+\int_{\Lambda_\epsilon^c \cap [tu^+ > a_0]}F_1(tu^{+})\,dx.
$$
Since $u \in H_\epsilon$, we know that
$$
\int_{[tu^+ \geq a_0]}|u^+|^2\,dx \leq \int_{{\mathbb{R}}^N}|u^+|^2\,dx=D,
$$
and so,
$$
|[tu^+ \geq a_0]| \leq \frac{D}{a^2_0}t^2=D_1t^2.
$$
By the definition of $F_1$,
$$
F_1(t) \leq a_1t^2+b_1, \quad \forall t \geq 0,
$$
then
$$
\int_{\Lambda_\epsilon^c\cap [tu^+ > a_0]}F_1(tu^+)\,dx \leq \int_{[tu^+ > a_0]}F_1(tu^+)\,dx \leq At^{2}.
$$
Hence,
$$
I_{\epsilon}(t u) \leq \frac{t^2}{2}\|u\|_\epsilon^{2}-\frac{1}{2}\int_{\Lambda_\epsilon \cup [tu^+ \leq a_0]}|tu^+|^2\log(|tu^+|^2)\,dx+At^{2},
$$
or equivalently,
$$
I_{\epsilon}(t u) \leq \frac{t^2}{2}\|u\|_\epsilon^{2}-t^2\int_{\Lambda_\epsilon \cup [tu^+ \leq a_0]}(|u^+|^2\log(t)+|u^+|^2\log(u^+))\,dx+At^{2}.
$$
From this,
$$
I_{\epsilon}(t u) \leq t^{2}\left(\frac{1}{2}\|u\|_\epsilon^{2}-\log(t)\int_{\Lambda_\epsilon \cup [tu^+ \leq a_0]}|u^+|^2-\frac{1}{2}\int_{\Lambda_\epsilon \cup [tu^+ \leq a_0]}|u^+|^2\log(|u^+|^2)\,dx+A\right).
$$
Since,
$$
\int_{\Lambda_\epsilon \cup [tu^+ \leq a_0]}|u^+|^2\,dx \geq \int_{\Lambda_\epsilon }|u^+|^2\,dx>0
$$
we derive that
$$
 I_{\epsilon}(t u)\leq t^{2}\left(\frac{1}{2}\|u\|_\epsilon^{2}-\log(t)\int_{\Lambda_\epsilon}|u^+|^2-\frac{1}{2}\int_{\Lambda_\epsilon \cup [tu^+ \leq a_0]}|u^+|^2\log(|u^+|^2)\,dx+A\right), \quad \forall t \geq 1.
$$
On the other hand, using the fact that  $I_\epsilon(u)<+\infty$, it follows that $|u^+|^2\log|u^+|^2\in L^{1}({\mathbb{R}}^N)$. Hence,
$$
\left|\int_{\Lambda_\epsilon \cup [tu^+ \leq a_0]}|u^+|^2\log(|u^+|^2)\,dx\right| \leq  \int_{{\mathbb{R}}^N}||u^+|^2\log(|u^+|^2)|\,dx<+\infty, \quad \forall t \geq 1.
$$
Thereby, setting
$$
C=\sup_{t \geq 1}\left(-\int_{\Lambda_\epsilon \cup [tu^+ \leq a_0]}|u^+|^2\log(|u^+|^2)\,dx \right)< +\infty,
$$
we obtain
$$
I_{\epsilon}(tu)\leq t^{2}\left(\frac{1}{2}\|u\|_\epsilon^{2}-\log(t)\int_{\Omega_1}|u^+|^2+C+A\right), \quad \forall t \geq 1,
$$
from where it follows that
$$
I_{\epsilon}(tu) \to -\infty \quad \mbox{as} \quad t \to +\infty.
$$
\end{proof}
From Lemma \ref{vicente} can define the minimax level
\begin{equation}\label{mountain}
c_{\epsilon}=\inf_{\gamma\in \Gamma_{\epsilon}}\max_{t\in [0, 1]}I_{\epsilon}(\gamma (t)), \quad\text{where}\,\, \Gamma_{\epsilon}=\{\gamma\in C([0, 1], H_{\epsilon}): \gamma(0)=0, I_{\epsilon}(\gamma(1))<0\}.
\end{equation}
Using a version of the mountain pass theorem without $(PS)$ condition (see \cite{AlvesdeMorais}), there is a Palais-Smale sequence
$(u_{n})$ at the level $c_{\epsilon}$, that is, $I_{\epsilon}(u_{n})\rightarrow c_{\epsilon}$ and
\begin{eqnarray*}
&&\int \big (\nabla u_{n} \nabla(v-u_{n})+(V(\epsilon x) +1)u_{n}(v-u_{n})\big)dx-\int  G'_{2}(\epsilon x, u_{n}^{+})(v-u_{n})dx\\
&&+\int F_{1}(v^{+})dx-\int F_{1}(u_{n}^{+})dx\geq -\tau_{n}\Vert v-u_{n}\Vert_{\epsilon}, \,\, \forall v\in H_{\epsilon}.
\end{eqnarray*}
In order to show the boundedness of $(PS)$ sequence of $I_\epsilon$, we will use the following logarithmic inequality, whose the proof can be found in del Pino and Dolbeault \cite[pg 153]{dJ}.

\begin{lemma} \label{LI} ( {\bf A new logarithmic inequality} ) \label{P1} There are constants $A,B>0$ such that
	$$
	\int |u|^2\log(|u|^2)\, dx \leq A+ B\log(\|u\|), \quad \forall  u \in H^{1}(\mathbb{R}^N) \setminus \{0\}.
	$$
\end{lemma}

As an immediate consequence we have the corollary

\begin{corollary}  \label{C1} There are $C,R>0$ such that if $u \in H^{1}(\mathbb{R}^N)$ and $\|u\| \geq R$, then
	$$
	\int \log(|u|^2)|u|^2\,dx \leq C (1+ \|u\|).
	$$
\end{corollary}

By the definition of $G_2$, it is easy to see that
$$
G_2(x,s) \leq F_2(s), \quad s \geq 0.
$$
Consequently
\begin{align}\label{ine}
I_{\epsilon}(u) \geq \tilde{J}_{\epsilon}(u)=\frac{1}{2}\|u\|_{\epsilon}^{2}+\displaystyle \int F_1(u^+)\,dx- \displaystyle \int  F_2(u^+)dx, \quad \forall u \in H_{\epsilon}.
\end{align}

\begin{lemma}\label{boundedness}
 Let $(v_n) \subset H_{\epsilon}$ be a sequence such that $(I_{\epsilon}(v_n))$ is bounded in $\mathbb{R}$. Then, $(v_n)$ is a bounded sequence in $H_{\epsilon}$.
\end{lemma}	
\begin{proof}	 By the assumption, there is $M>0$ such that
$$
M \geq I_{\epsilon}(v_n), \quad \forall n \in \mathbb{N}.
$$
Thus,
$$
M \geq \tilde{J}_{\epsilon}(v_n) = \displaystyle \frac{1}{2}\|u_n\|_{\epsilon}^{2}+\displaystyle \int F_1(v^+_n)\,dx- \displaystyle \int  F_2(v^+_n)dx, \\
$$
that is
$$
M \geq \displaystyle \frac{1}{2}\|v_n\|_{\epsilon}^{2}-\frac{1}{2}\int |v^+_n|^{2}\log(|v^+_n|^{2})\,dx,
$$
from where it follows that
\begin{equation} \label{Z0}
\|v_n\|_{\epsilon}^{2} \leq 2M+\int |v^+_n|^{2}\log(|v^+_n|^{2})\,dx, \quad \forall n \in \mathbb{N}.
\end{equation}
Without lost of generality we will assume that $v^+_n \not=0$, because otherwise, we have that inequality
$$
\|v_n\|_{\epsilon}^2 \leq 2M.
$$
From this, assume that there is $n \in \mathbb{N}$ such that $\|v^+_n\|_{\epsilon} \geq R$. By Corollary \ref{C1},
$$
\|v_n\|_{\epsilon}^{2} \leq 2M + C(1+\|v^+_n\|_{\epsilon}) \leq 2M + C(1+\|v_n\|_{\epsilon}).
$$
If $0<\|v^+_n\|_{\epsilon} \leq R$, Lemma \ref{P1} combine with (\ref{Z0}) to give
$$
\|v_n\|_{\epsilon}^{2}\leq 2M+A +B\log(R).
$$
The above analysis ensures that $(v_n)$ is bounded. $\square$
\end{proof}
\vspace{0.5 cm}

As a byproduct of the last lemma we have the boundedness of $(PS)_{c_{\epsilon}}$ sequences for $I_{\epsilon}$.

\begin{corollary}\label{coro}
 If $(v_n)$ is a $(PS)$ sequence for $I_{\epsilon}$, then $(v_n)$ is bounded in $H_\epsilon$.
	
\end{corollary}

\begin{lemma}\label{compactness} For any fixed $\epsilon>0$, let $(v_n) \subset H_\epsilon$ be a $(PS)_{d}$ sequence for $I_{\epsilon}$. Then, for each $\zeta>0$, there is a
number $R=R(\zeta)>0$ such that
$$
\underset{n\rightarrow\infty}{\lim\sup}\int_{\mathbb{R}^{N}\backslash B_{R}(0)}(\vert \nabla v_{n}\vert^{2}+(V(\epsilon x)+1)\vert v_{n}\vert^{2})dx\leq \zeta.
$$
\end{lemma}	
\begin{proof}
Let $\phi_{R}\in C^{\infty}(\mathbb{R}^{N}, \mathbb{R})$ be a cut-off function
such that
\begin{equation*}
\phi_{R}=0\quad x\in B_{R/2}(0), \quad \phi_{R}=1\quad x\in B^{c}_{R}(0),\quad 0\leq \phi_{R}\leq 1,\quad\text{and}\, \quad\vert \nabla\phi_{R}\vert\leq C/R,
\end{equation*}
where $C>0$ is a constant independent of $R$. Since the sequence $(\phi_{R}v_{n})$ is bounded in $H_{\epsilon}$,  we derive that
\begin{equation*}
I_{\epsilon}'( u_{n})(\phi_{R} v_{n})=o_{n}(1),
\end{equation*}
that is
\begin{eqnarray*}
\int\Big(\vert \nabla v_{n}\vert^{2}+(V(\epsilon x)+1)\vert v_{n}\vert^{2}\Big)\phi_{R}dx&=&\int_{\Lambda_{\epsilon}} F'_{2}(v_{n}^{+})\phi_{R}v_{n}dx+\int_{\mathbb{R}^N\backslash \Lambda_{\epsilon}} \widetilde{F}'_{2}(v_{n}^{+})\phi_{R}v_{n}dx\\
&&-\int  v_{n} \nabla v_{n}\nabla \phi_{R}dx -\int F'_{1}(v_{n}^{+})\phi_{R}v_{n}dx+o_{n}(1).
\end{eqnarray*}
Choosing $R>0$ such that $\Lambda_{\epsilon} \subset B_{R/2}(0)$, the H\"{o}lder inequality together with the boundedness of the sequence $(v_{n})$ in $H_{\epsilon}$ leads to
\begin{eqnarray*}
\int\Big(\vert \nabla v_{n}\vert^{2}+(V(\epsilon x)+1)\vert v_{n}\vert^{2}\Big)\phi_{R}dx&\leq&l\int \vert v_{n}\vert^{2}\phi_{R}dx+\frac{C}{R}\Vert v_{n}\Vert_{\epsilon}^{2}+o_{n}(1).
\end{eqnarray*}
So, fixing $\zeta>0$ and passing to the limit in the last inequality, it follows that
\begin{align*}
\underset{n\rightarrow\infty}{\lim\sup}\int_{\mathbb{R}^{N}\backslash B_{R}(0)}(\vert \nabla v_{n}\vert^{2}+(V(\epsilon x)+1)\vert v_{n}\vert^{2})dx\leq \frac{C}{R}< \zeta,
\end{align*}
for some $R$ sufficiently large.
\end{proof}

Our next lemma shows that $I_{\epsilon}$ verifies the $(PS)$ condition.

\begin{lemma} \label{compacidadeR} Let $(v_n)$ be a $(PS)_d$ sequence for $I_{\epsilon}$ with $
	v_n \rightharpoonup v $ in $H_{\epsilon}. $ Then, $v_n \to v$ in $H_{\epsilon}$. Moreover,
$$
F_1(v_n^{+}) \to F_1(v^{+}) \quad \mbox{and} \quad  F'_1(v_n^{+})v_n^{+} \to F'_1(v^{+})v^{+} \quad \mbox{in} \quad L^{1}({\mathbb{R}}^N).
$$
As a consequence, $v$ is a critical point of $I_\epsilon$ at level $d$, that is, $0 \in \partial I_\epsilon(v)$ and $I_\epsilon(v)=d$.

\end{lemma}

\begin{proof} Let $(v_n) \subset H_{\epsilon}$ be a $(PS)_{d}$ sequence for $I_{\epsilon}$.
By Corollary 1.2, the sequence $(v_n)$ is bounded in $H_{\epsilon}$, then without lost of generality we can assume that
$$
v_n \rightharpoonup v \quad \mbox{in} \quad H_{\epsilon}.
$$
By the last lemma, for any given $\zeta>0$, there is $R>0$ such that
$$
\limsup_{n \to +\infty}\int_{|x|>R}(|\nabla v_n|^{2}+(V(\epsilon x)+1)|v_n|^{2})\,dx<\zeta.
$$
Since $G'_2$ has a subcritical growth, the above estimate ensures that
$$
\int G'_2(\epsilon x, v_n^{+})w\,dx \to \int G'_2(\epsilon x, v^{+})w\,dx, \quad \forall w \in C_{0}^{\infty}(\mathbb{R}^N),
$$
$$
\int G'_2(\epsilon x, v_n^{+})v_n^{+}\,dx \to \int G'_2(\epsilon x, v^{+})v^{+}\,dx,
$$
and
$$
\int G_2(\epsilon x,v_n^{+})\,dx \to \int G_2(\epsilon x, v^{+})\,dx.
$$
Now, recalling that $I_{\epsilon}'(v_n)w=o_n(1)\|u_n\|_{\epsilon}$ for all $w \in C_{0}^{\infty}(\mathbb{R}^N)$, we deduce that $I_{\epsilon}'(v)w=0$ for all $w \in C_{0}^{\infty}(\mathbb{R}^N)$, and so, $I_{\epsilon}'(v)v=0$, that is,
$$
\int \big (|\nabla v|^2+(V(\epsilon x) +1)|v|^2\big)dx+\int F'_1(v^{+})v^{+}\,dx=\int G'_2(\epsilon x, v^{+})v^{+}\,dx.
$$
Moreover, the limit $I_{\epsilon}'(v_n)v_n=o_n(1)\|u_n\|_{\epsilon}$, that is,
$$
\displaystyle \int  \big (|\nabla v_n|^2+(V(\epsilon x) +1)|v_n|^2\big)dx+\int F'_1(v_n^{+})v_n^{+}\,dx= \displaystyle \int  G'_2(\epsilon x, v_n^{+})v_n^{+} dx +o_n(1).
$$
Gathering the above information, we deduce that
$$
\displaystyle \int \big (|\nabla v_n|^2+(V(\epsilon x) +1)|v_n|^2\big)dx+\int F'_1(v_n^{+})v_n^{+}\,dx= \displaystyle \int \big (|\nabla v|^2+(V(\epsilon x) +1)|v|^2\big)dx+\int F'_1(v^{+})v^{+}\,dx+o_n(1),
$$
from where it follows that, for some subsequence,
$$
v_n \to v \quad \mbox{in} \quad H_{\epsilon}
$$
and
$$
F'_1(v_n^{+})v_n^{+} \to F'_1(v^{+})v^{+} \quad \mbox{in} \quad L^{1}({\mathbb{R}}^N).
$$
Since $F_1$ is convex, even and $F(0)=0$, we know that $F'_1(t)t \geq F_1(t)\geq 0$ for all $t \in {\mathbb{R}}$. Thus, the last limit together with Lebesgue's theorem yields
$$
F_1(v_n^{+}) \to F_1(v^{+}) \quad \mbox{in} \quad L^{1}({\mathbb{R}}^N).
$$
The above limits permit to conclude that $0 \in \partial I_\epsilon(v)$ and $I_\epsilon(v)=d$.
\end{proof}

\begin{theorem}\label{teorema1}
The functional $I_{\epsilon}$ has a positive critical point $u_{\epsilon}\in H_{\epsilon}$ such that $I_{\epsilon}(u_{\epsilon})=c_{\epsilon}$,
where $c_{\epsilon}$ denotes the mountain pass level associated with $I_{\epsilon}$.
\end{theorem}
\begin{proof}
The existence of the critical point $u_{\epsilon}$ is an immediate result of Lemma \ref{vicente}, Corollary \ref{coro} and Lemma \ref{compacidadeR}.
The function $u_{\epsilon}$ is nonnegative, because
$$
I'_{\epsilon}(u_{\epsilon})(u_{\epsilon}^{-})=0\Rightarrow u_{\epsilon}^{-}=0,
$$
where $u_{\epsilon}^{-}=\min \{u_{\epsilon}, 0 \}$.  By a slight variant of the argument in \cite[Section 3.1]{ASZ} it follows from the maximum principle(see \cite[Theorem 1]{V}) that
$u_{\epsilon}(x)>0$ for a.e. $x\in {\mathbb{R}}^N$.
\end{proof}
In the sequel, we denote by $\mathcal{N}_{\epsilon}$ the set
$$
\mathcal{N}_{\epsilon}=\left\{u \in D(I_{\epsilon})\backslash \{0\}: I'_{\epsilon}(u)u=0  \right\}.
$$
The following lemma is important for the proof of Lemma \ref{lemma1}.

\begin{lemma} \label{lema} Assume that hypotheses $(V1)-(V2)$ are satisfied.  For each $u\in H_\epsilon^{+}$, let $g_{u}: \mathbb{R}^{+}\rightarrow \mathbb{R}$ be given by $g_{u}(t)=I_{\epsilon}(tu)$. Then
there exists a unique $t_{u}>0$ such that $g'_{u}(t)>0$ in $(0, t_{u})$ and $g'_{u}(t)<0$ in $(t_{u}, \infty)$.
\end{lemma}
\begin{proof}
As in the proof of Lemma \ref{vicente}, we have $g_{u}(0)=0$, $g_{u}(t)>0$ for $t>0$ small and $g_{u}(t)<0$ for $t>0$ large. Therefore,
$\max_{t\geq 0}g_{u}(t)$ is achieved at a global maximum point $t=t_{u}>0$ verifying $g'_{u}(t_{u})=0$ and $t_{u}u\in \mathcal{N}_{\epsilon}$.
Now we claim that $t_{u}>0$ is unique. Indeed, suppose that there exist $t_{2}>t_{1}>0$ such that $g'_{u}(t_{1})=g'_{u}(t_{2})=0$. Then, for $i=1, 2$,
\begin{equation*}
t_{i}\int \big (|\nabla u|^2+(V(\epsilon x) +1)|u|^2\big)dx-\int_{\Lambda_{\epsilon}} F'_{2}(t_{i}u^{+})u^{+}dx-\int_{\mathbb{R}^N\backslash \Lambda_{\epsilon}} \widetilde{F}'_{2}(t_{i}u^{+})u^{+}dx+\int  F'_{1}(t_{i}u^{+})u^{+}dx=0.
\end{equation*}
Hence,
\begin{equation*}
\int \big (|\nabla u|^2+(V(\epsilon x) +1)|u|^2\big)dx=\int_{\Lambda_{\epsilon}} \frac{F'_{2}(t_{i}u^{+})u^{+}}{t_{i}}dx+\int_{\mathbb{R}^N\backslash \Lambda_{\epsilon}} \frac{\widetilde{F}'_{2}(t_{i}u^{+})u^{+}}{t_{i}}dx-\int \frac{F'_{1}(t_{i}u^{+})u^{+}}{t_{i}}dx,
\end{equation*}
which implies that
\begin{eqnarray*}
&&\int_{\Lambda_{\epsilon}} \Big(\frac{F'_{2}(t_{2}u^{+})u^{+}}{t_{2}}-\frac{F'_{2}(t_{1}u^{+})u^{+}}{t_{1}}\Big)dx+\int_{\mathbb{R}^N\backslash \Lambda_{\epsilon}}\Big( \frac{\widetilde{F}'_{2}(t_{2}u^{+})u^{+}}{t_{2}}-\frac{\widetilde{F}'_{2}(t_{1}u^{+})u^{+}}{t_{1}}\Big)dx\\
&= &\int \Big(\frac{F'_{1}(t_{2}u^{+})u^{+}}{t_{2}}-\frac{F'_{1}(t_{1}u^{+})u^{+}}{t_{1}}\Big)dx.
\end{eqnarray*}
Since $u\in H_\epsilon^{+}$, the left side of above equality is positive. For the right side of above equality, we have
\begin{eqnarray*}
 \int \Big(\frac{F'_{1}(t_{2}u^{+})u^{+}}{t_{2}}-\frac{F'_{1}(t_{1}u^{+})u^{+}}{t_{1}}\Big)dx&=&\int_{u^{+}<\frac{\delta}{t_{2}}} \Big(\frac{F'_{1}(t_{2}u^{+})u^{+}}{t_{2}}-\frac{F'_{1}(t_{1}u^{+})u^{+}}{t_{1}}\Big)dx\nonumber\\
 &&+\int_{\frac{\delta}{t_{2}}<u^{+}<\frac{\delta}{t_{1}}} \Big(\frac{F'_{1}(t_{2}u^{+})u^{+}}{t_{2}}-\frac{F'_{1}(t_{1}u^{+})u^{+}}{t_{1}}\Big)dx\nonumber\\
  &&+\int_{u^{+}>\frac{\delta}{t_{1}}} \Big(\frac{F'_{1}(t_{2}u^{+})u^{+}}{t_{2}}-\frac{F'_{1}(t_{1}u^{+})u^{+}}{t_{1}}\Big)dx\nonumber\\
 &=&\int_{u^{+}<\frac{\delta}{t_{2}}} (u^{+})^{2}\log(\frac{t_{1}}{t_{2}})^{2}dx+\int_{u^{+}>\frac{\delta}{t_{1}}} \Big(\frac{1}{t_{2}}-\frac{1}{t_{1}}\Big)2\delta u^{+} dx\nonumber\\
 &&+\int_{\frac{\delta}{t_{2}}<u^{+}<\frac{\delta}{t_{1}}} \Big((u^{+})^{2}\log \frac{t_{1}^{2}(u^{+})^{2}}{\delta^{2}}+2u^{+}(\frac{\delta}{t_{2}}-u^{+})
 \Big)dx.
\end{eqnarray*}
A direct computation shows that the right side of the last last equality is negative, which is a contradiction and $t_{u}>0$ is unique.
\end{proof}
\noindent {\bf Remark 3.1} \, By Lemma \ref{lema}, for each $u\in H_\epsilon^{+}$, there is a unique $m_{\epsilon}(u)\in \mathcal{N}_{\epsilon}$. On the other hand,
if $u\in \mathcal{N}_{\epsilon}$, then $u\in H_\epsilon^{+}$. Otherwise, we have $\vert \text{supp}(u^{+})\cap \Lambda_{\epsilon}\vert=0$ and
\begin{eqnarray*}
\int \big (|\nabla u|^2+(V(0) +1)|u|^2\big)dx\leq\Vert u\Vert_{\epsilon}^{2}&=&\int_{\mathbb{R}^N\backslash \Lambda_{\epsilon}} \widetilde{F}'_{2}(u^{+})u^{+} dx-\int  F'_{1}(u^{+})u^{+} dx\\
&\leq& l\int |u|^2dx,
\end{eqnarray*}
which is impossible since $V(0) +1\geq 2l>0$ and $u\neq 0$.

Related to $\epsilon=0$, for simplicity, we shall assume that $0\in \Lambda$, $V(0)=V_{0}>-1$ and consider the problem
\begin{equation}\label{constant}
\left\{
\begin{array}{lc}
-\Delta u+ V_{0} u=u \log u^2, & \quad \mbox{in} \quad \mathbb{R}^{N}, \\
u \in H^1(\mathbb{R}^{N}). & \;  \\
\end{array}
\right.
\end{equation}
The corresponding energy functional associated to (\ref{constant}) will be denoted by \linebreak $J_0:H^{1}(\mathbb{R}^N) \rightarrow (-\infty, +\infty]$ and defined as
$$
J_{0}(u)=\dis\frac{1}{2}\int \big (|\nabla u|^2+(V_{0} +1)|u|^2\big)dx-\dis\frac{1}{2}\int u^2\log u^2dx.
$$

In \cite{sz} is proved that problem (\ref{constant}) has a positive ground state solution given by
\begin{equation}\label{sisi}
c_{0}:=\inf_{u \in \mathcal{N}_{0}}J_{0}(u)=\inf_{u \in D(I_{0})\backslash \{0\}}\Big\{\max_{t\geq 0}J_{0}(tu)\Big\},
\end{equation}
where
$$
\mathcal{N}_{0}=\left\{u \in D(J_{0})\backslash \{0\}; J_{0}(u)=\displaystyle \frac{1}{2} \int |u|^{2}\,dx  \right\}
$$
and
$$
D(J_{0})=\left\{u \in H^{1}(\mathbb{R}^N)\,:\, J_{0}(u)<+\infty \right\}.
$$

The next lemma shows that the mountain pass level $c_{\epsilon}$ in (\ref{mountain}) is the ground state energy for the functional $I_{\epsilon}$,
it also establishes an important relation between $c_{\epsilon}$ and $c_{0}$.

\begin{lemma} \label{lemma1}
	\noindent (a)  $c_{\epsilon}>0$, for $\epsilon>0$;

	\noindent (b) $c_{\epsilon}=\displaystyle \inf_{u \in \mathcal{N}_{\epsilon}}I_{\epsilon}(u)$, for $\epsilon\geq 0$;

     \noindent (c) $\displaystyle  \lim_{\epsilon\rightarrow 0}c_{\epsilon}= c_{0}$.
\end{lemma}
\begin{proof}
\noindent (a): Follows from Lemma \ref{vicente} (ii). \\

\noindent (b): Let $u\in \mathcal{N}_{\epsilon}$ and let us consider $I_{\epsilon}(t^{*}u)<0$, for some $t^{*}>0$. If
$\gamma_{\epsilon}:[0, 1]\rightarrow H_{\epsilon}$ is the continuous path $\gamma_{\epsilon}(t)=t\cdot t^{*}u$, then
\begin{equation}\label{she}
c_{\epsilon}\leq \max_{t \in [0, 1]}I_{\epsilon}(\gamma_{\epsilon}(t))\leq \max_{t\geq 0}I_{\epsilon}(tu)=I_{\epsilon}(u)
\end{equation}
and consequently $c_{\epsilon}\leq  \displaystyle  \inf_{u \in \mathcal{N}_{\epsilon}}I_{\epsilon}(u)$.\\

Now we prove the reverse inequality. By Theorem \ref{teorema1}, there exits  $u_{\epsilon}\in H_{\epsilon}$ with $u_{\epsilon}(x)>0$ for all $x\in \mathbb{R}^N$ such that
\begin{equation*}
I_{\epsilon}(u_{\epsilon})=c_{\epsilon}\quad\text{and} \quad 0 \in \partial I_{\epsilon}(u_{\epsilon}).
\end{equation*}
Then $u_{\epsilon}\in \mathcal{N}_{\epsilon}$, and so,
\begin{equation*}
\inf_{u \in \mathcal{N}_{\epsilon}}I_{\epsilon}(u)\leq I_{\epsilon}(u_{\epsilon})=c_{\epsilon}.
\end{equation*}

\noindent (c):By \cite[Theorem 1.2]{sz}, the infimum in (\ref{sisi}) is such that $c_{0}=J_{0}(u_{0})$, for some positive function $u_{0}\in \mathcal{N}_{0}$. Note that,
if $\varphi\in C_{0}^{\infty}(\mathbb{R}^N)$, $0\leq \varphi\leq 1$, $\varphi\equiv 1$ in $B_{1}(0)$ and $\varphi\equiv 0$ in $B_{2}(0)^{c}$, defining
$\varphi_{R}:=\varphi(\cdot/ R)$ and $u_{R}(x)=\varphi_{R}(x)u_{0}(x)$, we have that
\begin{equation*}
u_{R}\rightarrow u_{0}\,\, \text{in}\,\, H^{1}(\mathbb{R}^N)\,\, \text{as}\,\, R\rightarrow +\infty.
\end{equation*}
Fixing $R>0$ and arguing as in the proof of (\ref{she}), for a fixed $\epsilon>0$ we find
$$
c_{\epsilon}\leq \max_{t\geq 0}I_{\epsilon}(tu_{R})=I_{\epsilon}(t_{\epsilon}u_{R}),
$$
and
\begin{equation*}
\int \Big( \vert\nabla u_{R}\vert^{2}+ (V(\epsilon x)+1) u_{R}^{2}\Big)dx=\int_{\Lambda_{\epsilon}} \frac{F'_{2}(t_{\epsilon}u_R)u_R}{t_{\epsilon}}dx+\int_{\mathbb{R}^N\backslash \Lambda_{\epsilon}}\frac{\widetilde{F}'_{2}(t_{\epsilon}u_R)u_R}{t_{\epsilon}} dx-\int \frac{F'_{1}(t_{\epsilon}u_R)u_R}{t_{\epsilon}}dx.
\end{equation*}
Since $V(\epsilon x)\rightarrow V_{0}$ as $\epsilon\rightarrow 0$, by the Lebesgue Dominated Convergence theorem, we have from the left side of the above equality that
\begin{equation*}
\lim_{\epsilon\rightarrow 0}\int \Big( \vert\nabla u_{R}\vert^{2}+ (V(\epsilon x)+1) u_{R}^{2}\Big)dx=\int \Big( \vert\nabla u_{R}\vert^{2}+ (V_{0}+1) u_{R}^{2}\Big)\,dx.
\end{equation*}
Assuming $t_{\epsilon}\rightarrow +\infty$ as $\epsilon\rightarrow 0$, since $\Lambda_{\epsilon}\rightarrow \mathbb{R}^N$ as $\epsilon\rightarrow 0$,
it is easy to verify that the right side of the above equality goes to $+\infty$ as $\epsilon\rightarrow 0$, which is a contradiction. Thus, $(t_{\epsilon})$
is bounded in $\mathbb{R}$ for $\epsilon$ small enough.
Moreover, since
\begin{eqnarray*}
I_{\epsilon}(t_{\epsilon}u_{R})&=&J_{0}(t_{\epsilon}u_{R})+\frac{t_{\epsilon}^{2}}{2}\int (V(\epsilon x)-V_{0}) u_{R}^{2}dx\\
&&+\int_{\Lambda_{\epsilon}} F_{2}(t_{\epsilon}u_{R})dx+\int_{\mathbb{R}^N\backslash \Lambda_{\epsilon}} \widetilde{F}_{2}(t_{\epsilon}u_{R})dx-\int F_{2}(t_{\epsilon}u_{R})dx\\
&\leq&J_{0}(t_{R}u_{R})+\frac{t_{\epsilon}^{2}}{2}\int (V(\epsilon x)-V_{0}) u_{R}^{2}dx,\\
\end{eqnarray*}
where $t_{R}>0$ satisfies
$$
J_{0}(t_{R}u_{R})=\max_{t\geq 0}J_{0}(tu_{R}).
$$
Using $\displaystyle \sup_{x \in B_R(0)}|V(\epsilon x)-V_{0}| \rightarrow 0$  as $\epsilon\rightarrow 0$, we get
\begin{equation}\label{J}
\limsup_{\epsilon \rightarrow 0}c_{\epsilon}\leq \underset{\epsilon\rightarrow 0}{\lim\sup}\,I_{\epsilon}(t_{\epsilon}u_{R})\leq J_{0}(t_{R}u_{R}).
\end{equation}
Now, we use the fact that $(t_{R})$ is also bounded for $R$ large enough, $u_{R}\leq u_0$ and $F_{1}$ is increasing for $t\geq 0$ to deduce that
$$
F_{1}(t_{R} u_{R})\leq F_{1}(k u_{0}),
$$
for some $k>0$. Since $u_{0}\in \mathcal{N}_{0}$, we can ensure that $F_{1}(k u_{0})\in L^{1}(\mathbb{R}^N)$ for all $k\geq 0$. Thus, if
$R_{n}\rightarrow +\infty$ and $t_{R_{n}}\rightarrow t_{*}$, the Lebesgue Dominated Convergence theorem yields
$$
F_{1}(t_{R_{n}} u_{R_{n}})\rightarrow F_{1}(t_{*} u_{0})\quad\text{in}\,\, L^{1}(\mathbb{R}^N)
$$
and
$$
F'_{1}(t_{R_{n}} u_{R_{n}})t_{R_{n}} u_{R_{n}}\rightarrow F'_{1}(t_{*} u_{0})t_{*} u_{0}\quad\text{in}\,\, L^{1}(\mathbb{R}^N).
$$
As an immediate consequence, $t_{R}\rightarrow 1$ as $R\rightarrow +\infty$ and
$$
J_{0}(t_{R}u_{R})\rightarrow J_{0}(u_{0})\quad\text{as}\,\,R\rightarrow +\infty.
$$
This combined with (\ref{J}) gives
$$
\underset{\epsilon\rightarrow 0}{\limsup}\,c_{\epsilon}\,\leq J_{0}(u_{0})=c_{0}.
$$
Inasmuch as $I_{\epsilon}(u)\geq J_{\epsilon}(u)\geq J_{0}(u)$, $\forall \epsilon>0$, $u\in D(I_\epsilon)$, and by part (b) with $\epsilon=0$,
the reverse inequality holds:
$$
\underset{\epsilon\rightarrow 0}{\lim\inf}\,\, c_{\epsilon}\geq c_{0}.
$$
Therefore,
$$
\underset{\epsilon\rightarrow 0}{\lim}\,\, c_{\epsilon}= c_{0}.
$$
\end{proof}

\begin{lemma}\label{lemFat}
Let $(\omega_{n})\subset \mathcal{N}_{0}$ be a sequence satisfying $J_{0}(\omega_{n})\rightarrow c_{0}$ such that $\omega_{n} \rightharpoonup \omega$ with $\omega \not=0$ and $J'_{0}(\omega)\omega \leq 0$. Then,  $\omega_{n} \to \omega$ in $H^{1}(\mathbb{R}^N)$.
\end{lemma}
\begin{proof}
Since $J'_{0}(\omega)\omega \leq 0$, there is $t \in (0,1]$ such that $t\omega \in \mathcal{N}_{0}$, and so,
$$
c_{0} \leq J_0(t\omega)= \frac{t^{2}}{2} \int |\omega|^{2}\,dx \leq \liminf_{n\rightarrow\infty}\displaystyle \frac{1}{2} \int |\omega_{n}|^{2}\,dx \leq\limsup_{n\rightarrow\infty}\displaystyle \frac{1}{2} \int |\omega_{n}|^{2}\,dx=\lim_{n\rightarrow\infty}J_{0}(\omega_{n})=c_{0}.
$$
The  above argument yields  $t=1$ and $\omega_{n}\rightarrow \omega$ in $L^{2}(\mathbb{R}^N)$. For the case $N\geq 3$,
since $(\omega_{n})$ is bounded in $L^{2^{*}}(\mathbb{R}^N)$, by interpolation on the Lebesgue spaces, it follows that
$$
\omega_{n}\rightarrow \omega \quad\text{in}\,\, L^{p}(\mathbb{R}^N),\,\, \text{for all}\,\, 2\leq p<2^{*},
$$
and therefore $\int F'_{2}(\omega_{n})\omega_{n}dx\rightarrow  \int F'_{2}(\omega)\omega dx$.
For the case $N=1, 2$, for any $q>2$, since $(\omega_{n})$ is bounded in $L^{q}(\mathbb{R}^N)$, using interpolation on the Lebesgue spaces again, we have that
$$
\omega_{n}\rightarrow \omega \quad\text{in}\,\, L^{p}(\mathbb{R}^N),\,\, \text{for all}\,\, 2\leq p<q.
$$
Since $q>2$ is arbitrary, thus
$$
\omega_{n}\rightarrow \omega \quad\text{in}\,\, L^{p}(\mathbb{R}^N),\,\, \text{for all}\,\, 2\leq p<\infty.
$$
and $\int F'_{2}(\omega_{n})\omega_{n}dx\rightarrow  \int F'_{2}(\omega)\omega dx$.

Finally, using the equalities  $J'_{0}(\omega)\omega=0$ and
\begin{eqnarray*}
\int \Big(\vert \nabla \omega_{n}\vert^{2}+(V_{0}+1)\vert \omega_{n}\vert^{2}\Big)dx+\int  F'_{1}(\omega_{n})\omega_{n} dx =\int  F'_{2}(\omega_{n})\omega_{n} dx,
\end{eqnarray*}
we get
\begin{eqnarray*}
\int \Big(\vert \nabla \omega_{n}\vert^{2}+(V_{0}+1)\vert \omega_{n}\vert^{2}\Big)dx\rightarrow \int \Big(\vert \nabla \omega \vert^{2}+(V_{0}+1)\vert \omega \vert^{2}\Big)dx\quad\text{in}\,\,H^{1}(\mathbb{R}^N),
\end{eqnarray*}
from where it follows the desired result.

\end{proof}

\begin{lemma} \label{lea} Let $\epsilon_{n}\rightarrow 0$ and  $u_{n}\in H_{\epsilon_{n}}$ such that $I_{\epsilon_{n}}(u_{n})=c_{\epsilon_{n}}$
and $I'_{\epsilon_{n}}(u_{n})=0$. Then there exists the sequence $(y_{n})\subset \mathbb{R}^N$
such that $\psi_{n}(x)=u_{n}(x+y_{n})$ has a convergent subsequence in $H^{1}(\mathbb{R}^N)$.
Moreover, there is $y_{0}\in\Lambda$ such that
$$
\lim_{n\rightarrow\infty}\epsilon_{n}y_{n}=y_{0}\quad \text{and}\quad V(y_{0})=V_{0}.
$$
\end{lemma}
\begin{proof} \mbox{} Taking into account $u_{n}\in H_{\epsilon_{n}}$ such that $I_{\epsilon_{n}}(u_{n})=c_{\epsilon_{n}}$ and  Lemma \ref{boundedness}, it is
easy to see that $(u_{n})$ is bounded in $H_{\epsilon_{n}}$. Moreover, $(u_{n})$ is also bounded in $H^{1}(\mathbb{R}^N)$. Using \cite{lions},
there exist $r, \gamma>0$ and a sequence $(y_{n})\subset \mathbb{R}^N$ such that
\begin{equation}\label{lion1}
\underset{n\rightarrow\infty}{\lim\sup}\int_{B_{r}(y_{n})}\vert u_{n}(x)\vert^{2}dx\geq\gamma.
\end{equation}
Otherwise, we can conclude that
$$
u_{n}\rightarrow 0 \quad\text{in}\,\, L^{p}(\mathbb{R}^N),\,\, \text{for all}\,\, 2< p<2^{*}.
$$
and $\int F'_{2}(u_{n})u_{n}dx\rightarrow  0$, it yields that $I_{\epsilon_n}(u_{n})=c_{\epsilon_n} \rightarrow 0$ as $\rightarrow\infty$, which is a contradiction, since $c_{\epsilon_n} \to c_0>0$. Setting $\psi_{n}(x)=u_{n}(x+y_{n})$, then
 there is $\psi\in H^{1}(\mathbb{R}^N)\backslash \{0\}$
such that
\begin{equation}\label{weak}
\psi_{n}\rightharpoonup \psi\quad \text{in}\,\, H^{1}(\mathbb{R}^N)
\end{equation}
and
\begin{equation}\label{positive}
\int_{B_{r}(0)}\vert \psi(x)\vert^{2}dx\geq\gamma.
\end{equation}
In the sequel we will prove that the sequence $(\epsilon_{n}y_{n})$ is bounded. To this end, it is enough to show the following claim.

\begin{claim} \label{A1} $\underset{n\rightarrow\infty}{\lim}\text{dist}(\epsilon_{n}y_{n}, \overline{\Lambda})=0$.\\

\end{claim}

Indeed, if the claim does not hold, there exist $\delta>0$ and a subsequence of $(\epsilon_{n}y_{n})$, still denoted by itself, such that,
$$
\text{dist}(\epsilon_{n}y_{n}, \overline{\Lambda})\geq \delta, \quad \forall n\in \mathbb{N}.
$$		
Consequently, there is $r>0$ such that
$$
B_{r}(\epsilon_{n}y_{n})\subset \Lambda^{c},  \quad \forall n\in \mathbb{N}.
$$	
Using the fact that $\psi$ is a nonnegative function, there is a sequence of nonnegative functions $(\omega_{j})\subset H^{1}(\mathbb{R}^N)$
such that $\omega_{j}$ has a compact support in $\mathbb{R}^N$ and $\omega_{j}\rightarrow \psi$ in $H^{1}(\mathbb{R}^N)$ as $j\rightarrow \infty$. Now,
fixing $j>0$ and using $w_{j}$ as a test function, we have
\begin{equation}\label{test}
\int \Big( \nabla \psi_{n}\nabla \omega_{j} + (V(\epsilon_{n} x+ \epsilon_{n}y_{n})+1) \psi_{n} \omega_{j}\Big)dx=\int  G'_{2}(\epsilon_{n} x+ \epsilon_{n}y_{n}, \psi_{n})\omega_{j}dx-\int F'_{1}(\psi_{n})\omega_{j}dx.
\end{equation}
Note that
\begin{equation*}
\int G'_{2}(\epsilon_{n} x+ \epsilon_{n}y_{n}, \psi_{n})\omega_{j}dx=\int_{B_{\frac{r}{\epsilon_{n}}}(0)} G'_{2}(\epsilon_{n} x+ \epsilon_{n}y_{n}, \psi_{n})\omega_{j}dx+ \int_{\mathbb{R}^N\backslash B_{\frac{r}{\epsilon_{n}}}(0)} G'_{2}(\epsilon_{n} x+ \epsilon_{n}y_{n}, \psi_{n})\omega_{j}dx
\end{equation*}
and so,
\begin{equation*}
\int G'_{2}(\epsilon_{n} x+ \epsilon_{n}y_{n}, \psi_{n})\omega_{j}dx\leq l\int_{B_{\frac{r}{\epsilon_{n}}}(0)}\psi_{n}\omega_{j}dx+ \int_{\mathbb{R}^N\backslash B_{\frac{r}{\epsilon_{n}}}(0)} F'_{2}(\psi_{n})\omega_{j}dx.
\end{equation*}
Therefore,
\begin{eqnarray*}
&&\int \Big( \nabla \psi_{n}\nabla \omega_{j} + (V(\epsilon_{n} x+ \epsilon_{n}y_{n})+1) \psi_{n} \omega_{j}\Big)dx\\
&\leq &l\int_{B_{\frac{r}{\epsilon_{n}}}(0)}\psi_{n}\omega_{j}dx+ \int_{\mathbb{R}^N\backslash B_{\frac{r}{\epsilon_{n}}}(0)} F'_{2}(\psi_{n})\omega_{j}dx-\int F'_{1}(\psi_{n})\omega_{j}dx,
\end{eqnarray*}
implying that
\begin{eqnarray*}
\int_{\mathbb{R}^N}\Big( \nabla \psi_{n}\nabla \omega_{j} + A \psi_{n} \omega_{j}\Big)dx
\leq  \int_{\mathbb{R}^N\backslash B_{\frac{r}{\epsilon_{n}}}(0)} F'_{2}(\psi_{n})\omega_{j}dx,
\end{eqnarray*}
where $A=V_{0}+1-l>0$. As $\omega_{j}$ has a compact support in $\mathbb{R}^N$ and $\epsilon_{n}\rightarrow 0$, the boundedness of $(\psi_{n})$ imply that
 $$
 \int_{\mathbb{R}^N\backslash B_{\frac{r}{\epsilon_{n}}}(0)} F'_{2}(\psi_{n})\omega_{j}dx\rightarrow 0\quad\text{as}\quad n\rightarrow\infty,
 $$
 and
 \begin{eqnarray*}
\int\Big( \nabla \psi_{n}\nabla \omega_{j} + A\psi_{n} \omega_{j}\Big)dx\rightarrow \int\Big( \nabla \psi\nabla \omega_{j} + A\psi \omega_{j}\Big)dx, \quad\text{as}\quad n\rightarrow\infty.
\end{eqnarray*}
Hence,
 \begin{eqnarray*}
 \int \Big( \nabla \psi\nabla \omega_{j} + A\psi \omega_{j}\Big)dx\leq 0.
\end{eqnarray*}
Since $j$ is arbitrary, taking the limit of $j\rightarrow +\infty$, we obtain
 \begin{eqnarray*}
 \int \Big( \vert \nabla \psi\vert^{2} + A\vert \psi \vert^{2}\Big)dx= 0,
\end{eqnarray*}
which contradicts (\ref{positive}). This proves Claim 3.1.

From Claim 3.1, there is a subsequence of $(\epsilon_{n}y_{n})$ and $y_{0}\in \overline{\Lambda}$  such that
 $$
 \lim_{n\rightarrow\infty} \epsilon_{n}y_{n}=y_{0}.
 $$
\begin{claim} \label{A2} $y_{0}\in \Lambda$.\\

\end{claim}
Indeed, by using the definition of $G'_{2}$ and (\ref{test}), we have that
\begin{equation*}
\int \Big( \nabla \psi_{n}\nabla \omega_{j} + (V(\epsilon_{n} x+ \epsilon_{n}y_{n})+1) \psi_{n} \omega_{j}\Big)dx+ \int  F'_{1}(\psi_{n})\omega_{j}dx\leq \int F'_{2}(\psi_{n})\omega_{j}dx.
\end{equation*}
By using (\ref{weak}) and the fact that $\omega_{j}$ has a compact support, letting $n\rightarrow\infty$, we have
\begin{equation*}
\int \Big( \nabla \psi\nabla \omega_{j} + (V(y_{0})+1) \psi \omega_{j}\Big)dx+ \int  F'_{1}(\psi)\omega_{j}dx\leq \int F'_{2}(\psi)\omega_{j}dx.
\end{equation*}
Now, taking the limit of $j\rightarrow +\infty$, it yields that
\begin{equation*}
\int \Big( \vert \nabla \psi \vert^{2}+ (V(y_{0})+1) \vert \psi \vert^{2}\Big)dx+\int F'_{1}(\psi)\psi dx\leq \int F'_{2}(\psi)\psi dx.
\end{equation*}
Hence, there is $s_{1}\in (0, 1]$ such that
\begin{equation*}
s_{1}\psi\in \mathcal{N}_{V(y_{0})}=\{u\in H^{1}(\mathbb{R}^N)\backslash \{0\}: J'_{V(y_{0})}(u)u=0\},
\end{equation*}
where $J_{V(y_{0})}: H^{1}(\mathbb{R}^N)\rightarrow \mathbb{R}$ is given by
$$
J_{V(y_{0})}(u)=\displaystyle \frac{1}{2}\int \big (|\nabla u|^2+V(y_{0})|u|^2\big)dx-\frac{1}{2}\int u^{2}\log u^{2}dx.
$$
If $c_{V(y_{0})}$ denotes the mountain pass level associated with $J_{V(y_{0})}$, we must have
$$
c_{V(y_{0})}\leq J_{V(y_{0})}(s_{1} \psi)\leq \underset{n\rightarrow +\infty}{\lim\inf}\,I_{\epsilon_{n}}(u_{n})=\underset{n\rightarrow +\infty}{\lim\inf}\,c_{\epsilon_{n}}=c_{0}=c_{V(0)}.
$$
Thus,
$$
c_{V(y_{0})}\leq c_{V(0)},
$$
from where it follows that
$$
V(y_{0})\leq V(0)\equiv V_{0}.
$$
As $V_{0}=\displaystyle \inf_{x\in \overline{\Lambda}}V(x)$, the above inequality implies that
$$
V(y_{0})=V_{0}.
$$
Moreover, by $(V2)$, $y_{0}\not\in \partial \Lambda$. Then, $y_{0}\in \Lambda$ and the proof of Claim 3.2 is complete.

Now, we are going to prove that $\psi_{n}\rightarrow \psi$ in $H^{1}(\mathbb{R}^N)$.
Fixing $s_{n}>0$ such that $\tilde{\psi}_{n}=s_{n}\psi_{n}\in  \mathcal{N}_{0}$. By (\ref{ine}), we can see that
$$
c_{0}\leq J_{0}(\tilde{\psi}_{n})\leq \max_{t\geq 0}I_{\epsilon_{n}}(t\psi_{n})=I_{\epsilon_{n}}(u_{n})
$$
which together with Lemma \ref{lemma1} implies that $J_{0}(\tilde{\psi}_{n})\rightarrow c_{0}$. Since $(\psi_{n})$ and $(\tilde{\psi}_{n})$ are bounded in $H^{1}(\mathbb{R}^N)$ and $\tilde{\psi}_{n}\not\rightarrow 0$
in $H^{1}(\mathbb{R}^N)$, we deduce that for some subsequence, still denote by itself, $s_{n}\rightarrow s^{*}> 0$. Moreover, using that $u_n$ is a solution, we also have that $J'_0(\tilde{\psi})\tilde{\psi} \leq 0$. Since $\tilde{\psi} \not= 0$, by Lemma \ref{lemFat},
$$
\tilde{\psi}_{n}\rightarrow \tilde{\psi}\quad \text{in}\quad H^{1}(\mathbb{R}^N)
$$
or equivalently
$$
{\psi}_{n}\rightarrow {\psi}\quad \text{in}\quad H^{1}(\mathbb{R}^N),
$$
which finishes the proof.
\end{proof}
\begin{lemma}\label{le51}
Let $(\psi_{n})$ the sequence given in Lemma \ref{lea}. Then, $(\psi_{n}) \subset L^{\infty}(\mathbb{R}^N)$ and there is $K>0$
such that
$$
\vert \psi_{n}\vert_{\infty}\leq K, \quad \forall n\in \mathbb{N}.
$$
Moreover,
\begin{equation*}
  \psi_{n}(x)\rightarrow 0\quad\text{as}\quad \vert x\vert\rightarrow +\infty,
\end{equation*}
uniformly in $n\in \mathbb{N}$.
\end{lemma}

\begin{proof}
For any $R>0$, $0<r<\frac{R}{2}$, let $\eta\in C^{\infty}(\mathbb{R}^N)$, $0\leq\eta\leq 1$ with $\eta(x)=1$ if $\vert x\vert\geq R$ and $\eta(x)=0$
if $\vert x\vert\leq R-r$ and $\vert \nabla \eta\vert\leq\frac{2}{r}$. For $L>0$, let
\begin{equation*}
\psi_{L, n}:=\left\{
\begin{array}{l}
\psi_{n},\quad \quad \quad \text{if}\,\,\psi_{n}\leq L\\
L ,\quad\quad\quad\,\, \text{if}\,\,\psi_{n}\geq L.
\end{array}%
\right.
\end{equation*}
We first deal with the case $N>2$. To this end, let $z_{L, n}=\eta^{2}\psi_{L, n}^{2(\beta-1)}\psi_{n}$ and  $\omega_{L, n}=\eta\psi_{n}\psi_{L, n}^{\beta-1}$
with $\beta>1$ to be determined later. Since $0\leq\eta\leq 1$ and $\psi_{L, n}\leq L$, it yields that $\eta^{2}\psi_{L, n}^{2(\beta-1)}\psi_{n}\leq L^{2(\beta-1)}\psi_{n}$ and $F'_{1}(\psi_{n})\eta^{2}\psi_{L, n}^{2(\beta-1)}\psi_{n}\in L^{1}(\mathbb{R}^N)$. Taking
$z_{L, n}$ as a test function, we have
\begin{eqnarray}\label{ine31}
&&\int\eta^{2}\psi_{L, n}^{2(\beta-1)}\vert \nabla \psi_{n}\vert^{2}dx+\int(V(\epsilon_{n}x+\epsilon_{n}y_{n})+1)\eta^{2}\psi_{L, n}^{2(\beta-1)}\vert  \psi_{n}\vert^{2}dx+\int F'_{1}(\psi_{n})\eta^{2}\psi_{L, n}^{2(\beta-1)}\psi_{n}dx\nonumber\\
&=&-2\int \eta\psi_{L, n}^{2(\beta-1)}\psi_{n} \nabla \psi_{n}\nabla \eta dx-2(\beta-1)\int \eta^{2}\psi_{L, n}^{2\beta-3}\psi_{n} \nabla \psi_{n}\nabla \psi_{L, n} dx\nonumber\\
&&+\int G'_{2}(\epsilon_{n}x+\epsilon_{n}y_{n}, \psi_{n})\eta^{2}\psi_{L, n}^{2(\beta-1)}\psi_{n}dx.
\end{eqnarray}
From the definition of $G_{2}$,  we have that
\begin{eqnarray}\label{ine32}
G'_{2}( x, t)\leq lt+Ct^{2^{*}-1}, \quad  \forall (x, t)\in \mathbb{R}^N\times \mathbb{R}^{+}.
\end{eqnarray}
Using (\ref{eq2}), (\ref{ine31}) and (\ref{ine32}), we can obtain that
\begin{eqnarray}\label{ine33}
\int \eta^{2}\psi_{L, n}^{2(\beta-1)}\vert \nabla \psi_{n}\vert^{2}dx\leq 2\int \eta\psi_{L, n}^{2(\beta-1)}\psi_{n} \vert \nabla \psi_{n}\vert \vert\nabla \eta \vert dx
+C\int \eta^{2}\psi_{L, n}^{2(\beta-1)}\vert \psi_{n}\vert^{2^{*}}dx.
\end{eqnarray}
For each $\delta>0$, using Young's inequality, we have from (\ref{ine33}) that
\begin{eqnarray*}
\int \eta^{2}\psi_{L, n}^{2(\beta-1)}\vert \nabla \psi_{n}\vert^{2}dx&\leq& 2\delta\int \eta^{2}\psi_{L, n}^{2(\beta-1)} \vert \nabla \psi_{n}\vert^{2}dx+2C_{\delta}\int \psi_{L, n}^{2(\beta-1)}\vert \psi_{n}\vert^{2} \vert \nabla \eta\vert^{2} dx\\
&&+C\int_{\mathbb{R}^N}\eta^{2}\psi_{L, n}^{2(\beta-1)}\vert \psi_{n}\vert^{2^{*}}dx.
\end{eqnarray*}
Choosing $\delta\in (0, \frac{1}{4})$, it yields
\begin{eqnarray}\label{ine34}
\int \eta^{2}\psi_{L, n}^{2(\beta-1)}\vert \nabla \psi_{n}\vert^{2}dx\leq C\int \psi_{L, n}^{2(\beta-1)}\vert \psi_{n}\vert^{2} \vert \nabla \eta\vert^{2} dx+C\int \eta^{2}\psi_{L, n}^{2(\beta-1)}\vert \psi_{n}\vert^{2^{*}}dx.
\end{eqnarray}
On the other hand, by the Sobolev and H\"{o}lder inequalities, we have
\begin{eqnarray}\label{ine35}
\vert\omega_{L, n}\vert_{2^{*}}^{2}&\leq &C\int \vert \nabla\omega_{L, n}\vert^{2}dx=C\int \vert \nabla(\eta\psi_{n}\psi_{L, n}^{\beta-1})\vert^{2}dx\nonumber\\
&\leq& C\beta^{2}\Big(\int \psi_{L, n}^{2(\beta-1)}\vert \psi_{n}\vert^{2} \vert \nabla \eta\vert^{2} dx+\int \eta^{2}\psi_{L, n}^{2(\beta-1)}\vert \nabla\psi_{n}\vert^{2}dx\Big).
\end{eqnarray}
Combining (\ref{ine34}) and (\ref{ine35}), we have
\begin{eqnarray}\label{ine36}
\vert\omega_{L, n}\vert_{2^{*}}^{2}
\leq C\beta^{2}\Big(\int \psi_{L, n}^{2(\beta-1)}\vert \psi_{n}\vert^{2} \vert \nabla \eta\vert^{2} dx+\int \eta^{2}\psi_{L, n}^{2(\beta-1)}\vert \psi_{n}\vert^{2^{*}}dx\Big).
\end{eqnarray}
Let $\beta=\frac{2^{*}}{2}$, by the definition of $\omega_{L, n}$ and (\ref{ine36}), we rewrite the last inequality as
\begin{eqnarray*}
\Big(\int (\eta\psi_{n}\psi_{L, n}^{(2^{*}-2)/2})^{2^{*}}\Big)^{2/2^{*}}
&\leq& C(N, 2)\Big\{\Big(\int(\eta\psi_{n}\psi_{L, n}^{(2^{*}-2)/2} )^{2^{*}}dx\Big)^{2/2^{*}}\Big(\int_{\vert x\vert\geq R-r} \vert \psi_{n}\vert^{2^{*}}\Big)^{(2^{*}-2)/2}\\
&&+\int\psi_{L, n}^{2^{*}-2}\vert \psi_{n}\vert^{2} \vert \nabla \eta\vert^{2} dx\Big\}\\
&\leq& C(N, 2)\Big\{\Big(\int(\eta\psi_{n}\psi_{L, n}^{(2^{*}-2)/2} )^{2^{*}}dx\Big)^{2/2^{*}}\vert \psi_{n}\vert_{2^{*}(\vert x\vert\geq R/2)}^{2^{*}-2}\\
&&+\int\psi_{L, n}^{2^{*}-2}\vert \psi_{n}\vert^{2} \vert \nabla \eta\vert^{2} dx\Big\}.
\end{eqnarray*}
In view of $\psi_{n}\rightarrow \psi$ in $H^{1}(\mathbb{R}^N)$, for $R$ large enough, we conclude that
$$
\vert \psi_{n}\vert_{2^{*}(\vert x\vert\geq R/2)}^{2^{*}-2}\leq \frac{1}{2 C(N, 2)}\quad\text{uniformly in}\,\,n\in \mathbb{N}.
$$
Hence we obtain
\begin{eqnarray*}
\Big(\int_{\vert x\vert\geq R} (\psi_{n}\psi_{L, n}^{(2^{*}-2)/2})^{2^{*}}\Big)^{2/2^{*}}
&\leq& 2C(N, 2)\int \psi_{L, n}^{2^{*}-2}\vert \psi_{n}\vert^{2} \vert \nabla \eta\vert^{2} dx\\
&\leq& \frac{C}{r^{2}}\int \vert \psi_{n}\vert^{2^{*}}  dx.
\end{eqnarray*}
Using the Fatou's lemma in the variable $L$, we have
\begin{eqnarray}\label{ine37}
\psi_{n}\in L^{{{2^{*}}^{2}}/2}(\vert x\vert\geq R)\quad \text{for}\,\, R\,\,\, \text{large enough}.
\end{eqnarray}
Next, we note that if $\beta=2^{*}(t-1)/2t$ with $t={2^{*}}^{2}/2(2^{*}-2)$, then $\beta>1$ and
$2t/(t-1)<2^{*}$. Now suppose that $\psi_{n}\in L^{2\beta t/(t-1)}(\vert x\vert\geq R-r)$ for some $\beta\geq 1$.
Using the H\"{o}lder inequality with exponent $t/(t-1)$ and $t$, then (\ref{ine37}) gives that
\begin{eqnarray}\label{ine38}
\vert\omega_{L, n}\vert_{2^{*}}^{2}&\leq & C\beta^{2}\Big\{\Big(\int_{\vert x\vert\geq R-r}(\eta^{2}\psi_{n}^{2\beta})^{t/(t-1)}dx\Big)^{1-1/t}\Big(\int_{\vert x\vert\geq R-r} \vert \psi_{n}\vert^{(2^{*}-2)t}\Big)^{1/t}\nonumber\\
&&+\frac{(R^{N}-(R-r)^{N})^{1/t}}{r^{2}}\Big(\int_{\vert x\vert\geq R-r}\vert \psi_{n}\vert^{2\beta t/(t-1)}  dx\Big)^{1-1/t}\Big\}\nonumber\\
&\leq & C\beta^{2}\Big(1+\frac{R^{N/t}}{r^{2}}\Big)\Big(\int_{\vert x\vert\geq R-r}\vert \psi_{n}\vert^{2\beta t/(t-1)}  dx\Big)^{1-1/t}.
\end{eqnarray}
Letting $L\rightarrow +\infty$ in (\ref{ine38}), we obtain
$$
\vert \psi_{n}\vert_{2^{*}\beta(\vert x\vert\geq R)}^{2\beta}\leq  C\beta^{2}\Big(1+\frac{R^{N/t}}{r^{2}}\Big)\vert \psi_{n}\vert_{2\beta t/(t-1)(\vert x\vert\geq R-r)}^{2\beta}.
$$
If we set $\chi:=2^{*}(t-1)/(2t)$, $s:=2t/(t-1)$, then
\begin{eqnarray}\label{ine39}
\vert \psi_{n}\vert_{\beta \chi s(\vert x\vert\geq R)}\leq  C^{1/\beta}\beta^{1/\beta}\Big(1+\frac{R^{N/t}}{r^{2}}\Big)^{1/(2\beta)}\vert \psi_{n}\vert_{\beta s(\vert x\vert\geq R-r)}.
\end{eqnarray}
Let $\beta=\chi^{m}(m=1, 2, \cdots)$, we obtain
$$
\vert \psi_{n}\vert_{\chi^{m+1} s(\vert x\vert\geq R)}\leq  C^{\chi^{-m}}\chi^{m\chi^{-m}}\Big(1+\frac{R^{N/t}}{r^{2}}\Big)^{1/(2\beta)}\vert \psi_{n}\vert_{\chi^{m} s(\vert x\vert\geq R-r)}.
$$
It is clear that $2>N/t$. So if we take $r_{m}=2^{-(m+1)}R$, then (\ref{ine39}) implies
\begin{eqnarray*}
\vert \psi_{n}\vert_{\chi^{m+1} s(\vert x\vert\geq R)}&\leq & \vert \psi_{n}\vert_{\chi^{m+1} s(\vert x\vert\geq R-r_{m+1})}\\
&\leq &C^{\sum_{i=1}^{m}\chi^{-i}}\chi^{\sum_{i=1}^{m} i\chi^{-i}}\exp \Big( \sum_{i=1}^{m}\frac{\ln(1+2^{2(i+1)})}{2\chi^{i}}\Big)\vert \psi_{n}\vert_{\chi s(\vert x\vert\geq R-r_{1})}\\
&\leq &C\vert \psi_{n}\vert_{2^{*}(\vert x\vert\geq R/2)}.
\end{eqnarray*}
Letting $m\rightarrow\infty$ in the last inequality, we get
\begin{eqnarray}\label{ine40}
\vert \psi_{n}\vert_{L^{\infty}(\vert x\vert\geq R)}\leq C\vert \psi_{n}\vert_{2^{*}(\vert x\vert\geq R/2)}.
\end{eqnarray}
Using $\psi_{n}\rightarrow \psi$ in $H^{1}(\mathbb{R}^N)$ again, for any fixed $a>0$, there exists $R>0$ such that
$\vert \psi_{n}\vert_{L^{\infty}(\vert x\vert\geq R)}\leq a$ for all $n\in \mathbb{N}$. Therefore, $\underset{\vert x\vert\rightarrow\infty}{\lim} \psi_{n}(x)=0$
uniformly in $n$.

To show that $\vert \psi_{n}\vert_{L^{\infty}(\mathbb{R}^N)}<+\infty$, we need only show that for any $x_{0}\in \mathbb{N}$,
there is a ball $B_{R}(x_{0})=\{x\in \mathbb{R}^N): \vert x-x_{0}\vert \leq R\}$ such that $\vert \psi_{n}\vert_{L^{\infty}(B_{R}(x_{0}))}<+\infty$.
We can use the same arguments and take $\eta\in C^{\infty}(\mathbb{R}^N)$, $0\leq\eta\leq 1$ with $\eta(x)=1$ if $\vert x-x_{0}\vert\leq \rho'$ and $\eta(x)=0$
if $\vert x-x_{0}\vert> 2\rho'$ and $\vert \nabla \eta\vert\leq\frac{2}{\rho'}$, to prove that
\begin{eqnarray}\label{ine41}
\vert \psi_{n}\vert_{L^{\infty}(\vert x-x_{0}\vert\leq \rho')}\leq C\vert \psi_{n}\vert_{2(\vert x\vert\geq 2\rho')}.
\end{eqnarray}
From (\ref{ine40}) and (\ref{ine41}), using a standard covering argument it follows that
\begin{eqnarray*}
\vert \psi_{n}\vert_{L^{\infty}(\mathbb{R}^N)}\leq C
\end{eqnarray*}
for some positive constant $C$. \\

For the case $N=2$, similar with the proof for the case $N\geq 3$, we also let  $z_{L, n}=\eta^{2}\psi_{L, n}^{2(\beta-1)}\psi_{n}$ and  $\omega_{L, n}=\eta\psi_{n}\psi_{L, n}^{\beta-1}$
with $\beta>1$ to be determined later. Since $0\leq\eta\leq 1$ and $\psi_{L, n}\leq L$, it yields that $\eta^{2}\psi_{L, n}^{2(\beta-1)}\psi_{n}\leq L^{2(\beta-1)}\psi_{n}$ and $F'_{1}(\psi_{n})\eta^{2}\psi_{L, n}^{2(\beta-1)}\psi_{n}\in L^{1}(\mathbb{R}^N)$. Taking
$z_{L, n}$ as a test function, we have
\begin{eqnarray}\label{ine51}
&&\int\eta^{2}\psi_{L, n}^{2(\beta-1)}\vert \nabla \psi_{n}\vert^{2}dx+\int(V(\epsilon_{n}x+\epsilon_{n}y_{n})+1)\eta^{2}\psi_{L, n}^{2(\beta-1)}\vert  \psi_{n}\vert^{2}dx+\int F'_{1}(\psi_{n})\eta^{2}\psi_{L, n}^{2(\beta-1)}\psi_{n}dx\nonumber\\
&=&-2\int \eta\psi_{L, n}^{2(\beta-1)}\psi_{n} \nabla \psi_{n}\nabla \eta dx-2(\beta-1)\int \eta^{2}\psi_{L, n}^{2\beta-3}\psi_{n} \nabla \psi_{n}\nabla \psi_{L, n} dx\nonumber\\
&&+\int G'_{2}(\epsilon_{n}x+\epsilon_{n}y_{n}, \psi_{n})\eta^{2}\psi_{L, n}^{2(\beta-1)}\psi_{n}dx.
\end{eqnarray}
From the definition of $G_{2}$, for any $t\geq 0$ and $x\in \mathbb{R}^N$,  we have that
\begin{eqnarray}\label{ine52}
G'_{2}( x, t)\leq lt+Ct^{q-1},
\end{eqnarray}
where $2<q<\infty$.

By (\ref{eq2}), (\ref{ine51}) and (\ref{ine52}), we  obtain that
\begin{eqnarray}\label{ine53}
\int \eta^{2}\psi_{L, n}^{2(\beta-1)}\vert \nabla \psi_{n}\vert^{2}dx\leq 2\int \eta\psi_{L, n}^{2(\beta-1)}\psi_{n} \vert \nabla \psi_{n}\vert \vert\nabla \eta \vert dx
+C\int \eta^{2}\psi_{L, n}^{2(\beta-1)}\vert \psi_{n}\vert^{q}dx.
\end{eqnarray}
For any $\delta>0$, using Young's inequality, we have from (\ref{ine53}) that
\begin{eqnarray*}
\int \eta^{2}\psi_{L, n}^{2(\beta-1)}\vert \nabla \psi_{n}\vert^{2}dx&\leq& 2\delta\int \eta^{2}\psi_{L, n}^{2(\beta-1)} \vert \nabla \psi_{n}\vert^{2}dx+2C_{\delta}\int \psi_{L, n}^{2(\beta-1)}\vert \psi_{n}\vert^{2} \vert \nabla \eta\vert^{2} dx\\
&&+C\int_{\mathbb{R}^N}\eta^{2}\psi_{L, n}^{2(\beta-1)}\vert \psi_{n}\vert^{q}dx.
\end{eqnarray*}
Choosing $\delta\in (0, \frac{1}{4})$, it yields
\begin{eqnarray}\label{ine54}
\int \eta^{2}\psi_{L, n}^{2(\beta-1)}\vert \nabla \psi_{n}\vert^{2}dx\leq C\int \psi_{L, n}^{2(\beta-1)}\vert \psi_{n}\vert^{2} \vert \nabla \eta\vert^{2} dx+C\int \eta^{2}\psi_{L, n}^{2(\beta-1)}\vert \psi_{n}\vert^{q}dx.
\end{eqnarray}
On the other hand, by the Sobolev embedding,
\begin{eqnarray}\label{ine55}
\vert\omega_{L, n}\vert_{q}^{2}\leq C\beta^{2}\Big(\int \psi_{L, n}^{2(\beta-1)}\vert \psi_{n}\vert^{2} \vert \nabla \eta\vert^{2} dx+\int \eta^{2}\psi_{L, n}^{2(\beta-1)}\vert \nabla\psi_{n}\vert^{2}dx \Big).
\end{eqnarray}
Using (\ref{ine54}) and (\ref{ine55}), we have
\begin{eqnarray}\label{ine56}
\vert\omega_{L, n}\vert_{q}^{2}
\leq C\beta^{2}\Big(\int \psi_{L, n}^{2(\beta-1)}\vert \psi_{n}\vert^{2} \vert \nabla \eta\vert^{2} dx+\int \eta^{2}\psi_{L, n}^{2(\beta-1)}\vert \psi_{n}\vert^{q}dx\Big).
\end{eqnarray}
Let $\beta=\frac{q}{2}$, by the definition of $\omega_{L, n}$ and (\ref{ine36}), we rewrite the last inequality as
\begin{eqnarray*}
\Big(\int (\eta\psi_{n}\psi_{L, n}^{(q-2)/2})^{q}\Big)^{2/q}
&\leq& C(2, 2)\Big\{\Big(\int(\eta\psi_{n}\psi_{L, n}^{(q-2)/2} )^{q}dx\Big)^{2/q}\Big(\int_{\vert x\vert\geq R-r} \vert \psi_{n}\vert^{q}\Big)^{(q-2)/q}\\
&&+\int\psi_{L, n}^{q-2}\vert \psi_{n}\vert^{2} \vert \nabla \eta\vert^{2} dx\Big\}\\
&\leq& C(2, 2)\Big\{\Big(\int(\eta\psi_{n}\psi_{L, n}^{(q-2)/2} )^{q}dx\Big)^{2/q}\vert \psi_{n}\vert_{q(\vert x\vert\geq R/2)}^{q-2}\\
&&+\int\psi_{L, n}^{q-2}\vert \psi_{n}\vert^{2} \vert \nabla \eta\vert^{2} dx\Big\}.
\end{eqnarray*}

In view of $\psi_{n}\rightarrow \psi$ in $H^{1}(\mathbb{R}^N)$, we have $\psi_{n}\rightarrow \psi$ in $L^{q}(\mathbb{R}^N)$. Thus, for $R$ large enough, we conclude that
$$
\vert \psi_{n}\vert_{q(\vert x\vert\geq R/2)}^{q-2}\leq \frac{1}{2 C(2, 2)}\quad\text{uniformly in}\,\,n\in \mathbb{N}.
$$
Hence we obtain
\begin{eqnarray*}
\Big(\int_{\vert x\vert\geq R} (\psi_{n}\psi_{L, n}^{(q-2)/2})^{q}\Big)^{2/q}
&\leq& 2C(2, 2)\Big(\int\psi_{L, n}^{q-2}\vert \psi_{n}\vert^{2} \vert \nabla \eta\vert^{2} dx+\int \eta^{2} \psi_{L, n}^{q-2}\vert \psi_{n}\vert^{2} dx\Big)\\
&\leq& \frac{C}{r^{2}}\int \vert \psi_{n}\vert^{q}  dx.
\end{eqnarray*}
Using the Fatou's lemma in the variable $L$, we have
\begin{eqnarray}\label{ine57}
\psi_{n}\in L^{{{q}^{2}}/2}(\vert x\vert\geq R)\quad \text{for}\,\, R\,\,\, \text{large enough}.
\end{eqnarray}
Next, we note that if $\beta=q(t-1)/2t$ with $t={q^{2}}/2(q-2)$, then $\beta>1$ and
$2t/(t-1)<q$. Now suppose that $\psi_{n}\in L^{2\beta t/(t-1)}(\vert x\vert\geq R-r)$ for some $\beta\geq 1$.
Using the H\"{o}lder inequality with exponent $t/(t-1)$ and $t$, then (\ref{ine56}) gives that
\begin{eqnarray}\label{ine58}
\vert\omega_{L, n}\vert_{q}^{2}&\leq & C\beta^{2}\Big\{\Big(\int_{\vert x\vert\geq R-r}(\psi_{n}^{2\beta})^{t/(t-1)}dx\Big)^{1-1/t}\Big(\int_{\vert x\vert\geq R-r} \vert \psi_{n}\vert^{(q-2)t}\Big)^{1/t}\nonumber\\
&&+\frac{(R^{2}-(R-r)^{2})^{1/t}}{r^{2}}\Big(\int_{\vert x\vert\geq R-r}\vert \psi_{n}\vert^{2\beta t/(t-1)}  dx\Big)^{1-1/t}\Big\}\nonumber\\
&\leq & C\beta^{2}\Big(1+\frac{R^{2/t}}{r^{2}}\Big)\Big(\int_{\vert x\vert\geq R-r}\vert \psi_{n}\vert^{2\beta t/(t-1)}  dx\Big)^{1-1/t}.
\end{eqnarray}
Letting $L\rightarrow +\infty$ in (\ref{ine58}), we obtain
$$
\vert \psi_{n}\vert_{q\beta(\vert x\vert\geq R)}^{2\beta}\leq  C\beta^{2}\Big(1+\frac{R^{2/t}}{r^{2}}\Big)\vert \psi_{n}\vert_{2\beta t/(t-1)(\vert x\vert\geq R-r)}^{2\beta}.
$$
If we set $\chi:=q(t-1)/(2t)$, $s:=2t/(t-1)$, then
\begin{eqnarray}\label{ine39}
\vert \psi_{n}\vert_{\beta \chi s(\vert x\vert\geq R)}\leq  C^{1/\beta}\beta^{1/\beta}\Big(1+\frac{R^{2/t}}{r^{2}}\Big)^{1/(2\beta)}\vert \psi_{n}\vert_{\beta s(\vert x\vert\geq R-r)}.
\end{eqnarray}
Let $\beta=\chi^{m}(m=1, 2, \cdots)$, we obtain
$$
\vert \psi_{n}\vert_{\chi^{m+1} s(\vert x\vert\geq R)}\leq  C^{\chi^{-m}}\chi^{m\chi^{-m}}\Big(1+\frac{R^{2/t}}{r^{2}}\Big)^{1/(2\beta)}\vert \psi_{n}\vert_{\chi^{m} s(\vert x\vert\geq R-r)}.
$$
It is clear that $2>2/t$. So if we take $r_{m}=2^{-(m+1)}R$, then (\ref{ine39}) implies
\begin{eqnarray*}
\vert \psi_{n}\vert_{\chi^{m+1} s(\vert x\vert\geq R)}&\leq & \vert \psi_{n}\vert_{\chi^{m+1} s(\vert x\vert\geq R-r_{m+1})}\\
&\leq &C^{\sum_{i=1}^{m}\chi^{-i}}\chi^{\sum_{i=1}^{m} i\chi^{-i}}\exp \Big( \sum_{i=1}^{m}\frac{\ln(1+2^{2(i+1)})}{2\chi^{i}}\Big)\vert \psi_{n}\vert_{\chi s(\vert x\vert\geq R-r_{1})}\\
&\leq &C\vert \psi_{n}\vert_{q(\vert x\vert\geq R/2)}.
\end{eqnarray*}
Letting $m\rightarrow\infty$ in the last inequality, we get
\begin{eqnarray}\label{ine60}
\vert \psi_{n}\vert_{L^{\infty}(\vert x\vert\geq R)}\leq C\vert \psi_{n}\vert_{q(\vert x\vert\geq R/2)}.
\end{eqnarray}
Using $\psi_{n}\rightarrow \psi$ in $H^{1}(\mathbb{R}^2)$ again, for any fixed $a>0$, there exists $R>0$ such that
$\vert \psi_{n}\vert_{L^{\infty}(\vert x\vert\geq R)}\leq a$ for all $n\in \mathbb{N}$. Therefore, $\underset{\vert x\vert\rightarrow\infty}{\lim} \psi_{n}(x)=0$
uniformly in $n$.

Similarly, in order to show that $\vert \psi_{n}\vert_{L^{\infty}(\mathbb{R}^2)}<+\infty$, we need only show that for any $x_{0}\in \mathbb{R}^2$,
there is a ball $B_{R}(x_{0})=\{x\in \mathbb{R}^2: \vert x-x_{0}\vert \leq R\}$ such that $\vert \psi_{n}\vert_{L^{\infty}(B_{R}(x_{0}))}<+\infty$.
We can use the same arguments and take $\eta\in C^{\infty}(\mathbb{R}^2)$, $0\leq\eta\leq 1$ with $\eta(x)=1$ if $\vert x-x_{0}\vert\leq \rho'$ and $\eta(x)=0$
if $\vert x-x_{0}\vert> 2\rho'$ and $\vert \nabla \eta\vert\leq\frac{2}{\rho'}$, to prove that
\begin{eqnarray}\label{ine61}
\vert \psi_{n}\vert_{L^{\infty}(\vert x-x_{0}\vert\leq \rho')}\leq C\vert \psi_{n}\vert_{2(\vert x\vert\geq 2\rho')}.
\end{eqnarray}
From (\ref{ine60}) and (\ref{ine61}), using a standard covering argument it follows that
\begin{eqnarray*}
\vert \psi_{n}\vert_{L^{\infty}(\mathbb{R}^2)}\leq C
\end{eqnarray*}
for some positive constant $C$.

For the case $N=1$, we can deal with in the same way as the case $N=2$, we omit the details. Thus we have completed the proof of Lemma \ref{le51}.
\end{proof}

\section{Proof of Theorem \ref{teorema}}

By Theorem \ref{teorema1}, we know that problem (\ref{1'}) has a positive solution $u_{\epsilon}$ for all $\epsilon>0$.
On the other hand, by Lemma \ref{lea}, there  exists a sequence $(y_{n})\subset \mathbb{R}^N$ with $\displaystyle \lim_{n\rightarrow\infty}\epsilon_{n}y_{n}=y_{0}$ and $V(y_{0})=V_{0}$. Now, we can find $r>0$, such that $B_{r}(y_{n})\subset \Lambda$ for all $n\in N$.
Therefore $B_{r/\epsilon_{n}}(y_{n})\subset \Lambda_{\epsilon_{n}}$, $n\in N$. As a consequence
$$
\mathbb{R}^{N}\backslash \Lambda_{\epsilon_{n}}\subset \mathbb{R}^{N}\backslash B_{r/\epsilon_{n}}(y_{n})\quad\text{for any}\,\,n\in N.
$$
By using Lemma \ref{le51}, there exists $R>0$ such that
$$
  \psi_{n}(x)<a_{0}\quad\text{for}\,\,\vert x\vert>R, \,\,n\in N,
$$
where $\psi_{n}(x)=u_{\epsilon_{n}}(x+y_{n})$. Hence $u_{\epsilon_{n}}<a_{0}$ for any $x\in \mathbb{R}^{N}\backslash B_{R}(y_{n})$ and $n\in N$. Then there exists $n_{0}\in N$ such that for any $n\geq n_{0}$ and $r/\epsilon_{n} >R$ it holds
$$
\mathbb{R}^{N}\backslash \Lambda_{\epsilon_{n}}\subset \mathbb{R}^{N}\backslash B_{r/\epsilon_{n}}(y_{n})\subset \mathbb{R}^{N}\backslash B_{R}(y_{n}),
$$
which gives $u_{\epsilon_{n}}<a_{0}$ for any $x\in \mathbb{R}^{N}\backslash \Lambda_{\epsilon_{n}}$ and $n\geq n_{0}$.

This means that there exists $\epsilon_{0}>0$, problem (\ref{1'}) has a positive solution $u_{\epsilon}$ for all $\epsilon\in (0, \epsilon_{0})$. Taking
$v_{\epsilon}(x)=u_{\epsilon}(\frac{x}{\epsilon})$, we can infer that $v_{\epsilon}$ is a solution to problem $(P_\epsilon)$.

 Finally, we study
the behavior of the maximum points of $v_{\epsilon}(x)$. Take $\epsilon_{n}\rightarrow 0$ and $(u_{\epsilon_{n}})$ a sequence of solutions to problem (\ref{1'}). By the definition of $G_{2}$, there exists $\gamma\in (0, a_{0})$ such that
\begin{equation*}
G'_{2}(\epsilon x, t)t\leq lt^{2},\quad\text{for all}\,\, x\in \mathbb{R}^{N},\, 0\leq t\leq\gamma.
\end{equation*}
Using a similar argument above, we can take $R>0$ such that
\begin{equation}\label{5.14}
\Vert u_{\epsilon_{n}}\Vert_{L^{\infty}(B_{R}^{c}(y_{n}))}<\gamma.
\end{equation}
Up to a subsequence, we may also assume that
\begin{equation}\label{5.15}
\Vert u_{\epsilon_{n}}\Vert_{L^{\infty}(B_{R}(y_{n}))}\geq\gamma.
\end{equation}
Indeed, if \eqref{5.15} does not hold, we have $\Vert u_{\epsilon_{n}}\Vert_{L^{\infty}(\mathbb{R}^{N})}<\gamma$, and it follows from
$I'_{\epsilon_{n}}(u_{\epsilon_{n}})=0$ that
\begin{align*}
\int (\vert\nabla  u_{\epsilon_{n}}\vert^{2}+ (V_{0}+1)\vert u_{\epsilon_{n}}\vert^{2})dx\leq \Vert u_{\epsilon_{n}}\Vert^{2}_{\epsilon_{n}}\leq \int G'_{2}(\epsilon_{n} x,  u_{\epsilon_{n}}) u_{\epsilon_{n}}dx\leq l\int \vert u_{\epsilon_{n}}\vert^{2}dx.
\end{align*}
This fact shows  $u_{\epsilon_{n}}\equiv 0$ which is a contradiction. Hence \eqref{5.15} holds.

Taking into account \eqref{5.14} and \eqref{5.15}, we can infer that the global maximum points $p_{n}$ of $u_{\epsilon_{n}}$ belongs to $B_{R}(y_{n})$, that is
$p_{n}=q_{n}+y_{n}$ for some $q_{n}\in B_{R}$. Recalling that the associated solution of problem $(P_\epsilon)$ is of form $v_{n}(x)=u_{\epsilon_{n}}(x/\epsilon_{n})$ , we can see that
a  maximum point $\eta_{\epsilon_{n}}$ of $ \hat{v}_{n}$ is $\eta_{\epsilon_{n}}=\epsilon_{n}y_{n}+\epsilon_{n}q_{n}$. Since $q_{n}\in B_{R}$,
$\epsilon_{n}y_{n}\rightarrow y_{0}$ and $V(y_{0})=V_{0}$, from the continuity of $V$, we can conclude that
\begin{equation*}
\lim_{n\rightarrow\infty}V(\eta_{\epsilon_{n}})=V_{0},
\end{equation*}
which concludes the proof of the theorem.

\section*{Acknowledgements} The authors would like to thank the anonymous referees for their valuable suggestions and comments.

\vspace{1 cm}

\noindent \textsc{Claudianor O. Alves } \\
Unidade Acad\^{e}mica de Matem\'atica\\
Universidade Federal de Campina Grande \\
Campina Grande, PB, CEP:58429-900, Brazil \\
\texttt{coalves@mat.ufcg.edu.br} \\
\noindent and \\
\noindent \textsc{Chao Ji}(Corresponding Author) \\
Department of Mathematics\\
East China University of Science and Technology \\
Shanghai 200237, PR China \\
\texttt{jichao@ecust.edu.cn}

\end{document}